\documentclass[11pt]{article}
\usepackage[nohead,margin=1.0in]{geometry}
\usepackage{amssymb, amsmath, amsthm,amsfonts}
\usepackage{graphicx,epsfig}
\usepackage[tight]{subfigure}
\usepackage{cite}
\usepackage{times}
\usepackage{algpseudocode}
\usepackage{float}
\usepackage{appendix}
\usepackage{color}
\usepackage{epstopdf}
\usepackage{mathrsfs}
\usepackage{algorithm}
\usepackage{booktabs}
\usepackage{verbatim}
\usepackage[colorlinks,linktocpage,linkcolor=blue]{hyperref}

\usepackage{booktabs}
\usepackage{threeparttable}

\numberwithin{equation}{section}
\newtheorem{theorem}{Theorem}[section]
\newtheorem{remark}{Remark}[section]
\newtheorem{proposition}{Proposition}[section]
\newtheorem{lemma}{Lemma}[section]

\newtheorem{assumption}[theorem]{Assumption}

\allowdisplaybreaks

\newcommand{\red}[1]{{\color{black} #1}}

\def\al{\alpha}

\def\dx{\mathrm{d} x}
\def\dt{\mathrm{d} t}
\def\ds{\mathrm{d} s}
\def\dz{\mathrm{d} z}

\def\II{(\Omega)}
\def\la{\lambda}
\def\IPP{IPP}
\def\P01{P_\mathcal{A}}
\def\dtau{\overline{\partial}_{\tau}^{\alpha}}

\title{\red{Determining a Time-Varying} Potential in Time-Fractional Diffusion from Observation at a Single Point\thanks{The work of K. Shin is supported by Basic Science Research Program through the National Research Foundation of Korea
(NRF) funded by the Ministry of Education (Grant No. 2019R1A6A1A11051177) and the National Research Foundation of Korea(NRF) grant funded by the Korea government(MSIT)(RS-2024-00350215). The work of  Z. Zhou is supported by Hong Kong Research Grants Council (No. 15303021) and an internal grant of Hong Kong Polytechnic University (Project ID: P0038888, Work Programme: 1-ZVX3).}}

\author{Siyu Cen\thanks{Department of Applied Mathematics,
The Hong Kong Polytechnic University, Kowloon, Hong Kong, P.R. China (\texttt{siyu2021.cen@connect.polyu.hk}).}
\and Kwancheol Shin\thanks{Institute of Mathematical Science at Ewha Womans University, 52, Ewhayeodae-gil, Seodaemun-gu, Seoul
03760, Republic of Korea (\texttt{kcshin3623@gmail.com})}
\and Zhi Zhou\thanks{Department of Applied Mathematics,
The Hong Kong Polytechnic University, Kowloon, Hong Kong, P.R. China (\texttt{ zhizhou@polyu.edu.hk})}}

\begin{document}

\maketitle

\begin{abstract}
We discuss the identification of a time-dependent potential in a time-fractional diffusion model from a boundary measurement taken at a single point. Theoretically, we establish a conditional Lipschitz stability for this inverse problem. Numerically, we develop an easily implementable iterative algorithm to recover the unknown coefficient, and also derive rigorous error bounds for the discrete reconstruction. These results are attained by leveraging the (discrete) solution theory of direct problems, and applying error estimates that are optimal with respect to problem data regularity. Numerical simulations are provided to demonstrate the theoretical results.

\end{abstract}

\section{Introduction}\label{intro}
In this paper, we aim to study the numerical treatment for an inverse potential problem for \red{a time-fractional} diffusion model. Consider the convex polyhedral domain $\Omega \subset \mathbb{R}^d$, with $1 \le d \le 3$, and let $\partial\Omega$ denote the boundary of $\Omega$. Consider the initial-boundary value problem: 
 \begin{equation}\label{eqn:fde}
 \left\{\begin{aligned}
     \partial_t^\alpha u  -  \Delta u  + \rho(t)u &=f , &&\mbox{in } \Omega\times{(0,T]},\\
      \partial_\nu u &=0,&&\mbox{on } \partial\Omega\times{(0,T]},\\
    u(0)&=u_0 ,&&\mbox{in }\Omega,
  \end{aligned}\right.
 \end{equation}
where $T > 0$ denotes the prescribed final time, $u_0$ is the given initial condition, $f$ represents the source term which depends on the spatial variable, and $\nu$ is the outward unit normal vector \red{to the boundary} $\partial\Omega$. The notation $\partial^{\alpha}_t u $, with $\alpha\in(0,1)$, denotes the Djrbashian--Caputo fractional derivative of order $\alpha$, defined by 
\begin{equation}\label{eqn:frac_der}
    \partial^{\alpha}_t u(t) = \frac{1}{\Gamma (1 - \alpha)} \int_0^t (t-s)^{-\alpha} u'(s){\rm d}s,
\end{equation}
where $\Gamma(z) = \int_0^{\infty} s^{z-1}e^{-s}\ds$ with $\Re (z)>0$ denotes Gamma function.

In recent years, fractional evolution models have garnered considerable interest for their exceptional ability to characterize the anomalous diffusion phenomenon, which is prevalent across a wide spectrum of engineering and physical contexts. The array of successful applications is extensive and continuously expanding. Notable examples include the diffusion of proteins within cellular environments \cite{Golding:2006}, the movement of contaminants through groundwater systems \cite{Kirchner:2000}, and memory-dependent heat conduction processes \cite{Wolfersdorf:1994}, among others. For an in-depth exploration of the derivation of pertinent mathematical models and their numerous applications in the realms of physics and biology, readers are directed to thorough reviews \cite{MetzlerKlafter:2000,MetzlerJeon:2014} and detailed monographs \cite{Du:book,Jin:2021}.

In this paper, we address the inverse potential problem 
(\IPP) described as follows: for a fixed point $x_0\in \overline\Omega$, our objective is to reconstruct the potential function $\rho^{\dag}(t)$ from the single point measurement $u^{\dag}(x_0,t)$ with $t\in [0,T]$.
We let $u^{\dag}$ denote the exact solution with the exact potential $\rho^{\dag}$ that belongs to admissible set 
\begin{equation}\label{eqn:admissible_set}
    \mathcal{B} =\{\rho\in C[0,T]:0\le \rho \le \overline{c}_\rho\},
\end{equation}
where $\overline{c}_\rho>0$ is a constant.  
In practical scenarios, the actual measurement, denoted by $g_\delta(t)$, typically contains noise. We assume that $g_\delta(t)$ satisfies the following noise condition
\begin{equation}\label{eqn:noise}
\|g_\delta - u(x_0,t;\rho^\dag)\|_{C[0,T]} = \delta
\end{equation}
where $\delta$ denotes the noise level.

This research provides the following contributions. First of all, we establish a conditional Lipschitz stability estimate under some mild assumptions on problem data, \red{as given in }Theorem \ref{thm:stab}. 
This stability estimate leverages the smoothing properties of the direct problem and employs a carefully selected weighted $L^p$ norm.
Our second contribution, presented in Theorem \ref{thm:recon}, involves the development of a practical, fully discrete fixed-point algorithm. 
Moreover, we analyze the error for the discrete reconstruction $(\rho_*^n)_{n=1}^N$ in the $\ell^p$ norm:
\begin{equation}\label{eqn:error-intro}
  \| [\rho_*^n  -  \rho^\dag(t_n)]_{n=1}^N \|_{\ell^p} \le   c\left(\tau^{1/p}|\log\tau | + h^{2}|\log h|^3+\tau^{-\al} \delta \right)
\end{equation}
for any $p\in(1, \infty)$. This estimate provides clear guidance for selecting algorithmic parameters, such as spatial mesh size $h$ and temporal step size $\tau$, in relation to the noise level $\delta$. The outcome is achieved through the application of the weighted $\ell^p$ norm, combined with error estimates that are optimally aligned with the regularity of the problem data for the direct problem. 

The research into inverse problems for time-fractional evolution models began more recently, with significant contributions originating from \cite{ChengNakagawa:2009} (see \cite{LiYamamoto:2019a, LiuLiYamamoto:2019c} for some recent overviews) and numerous studies have focused on the reconstruction of a space-dependent potential or conductivity from lateral Cauchy data \cite{RundellYamamoto:2020,JingYamamoto:2021,KianLiLiu:2020,CenJinLiuZhou:2023,JinZhou:IP2021} or from the terminal data \cite{KaltenbacherRundell:2020,KaltenbacherRundell:2019,ZhangZhangZhou:2022,JinKianZhou:2023,JinLuQuanZhou:2024}. Although significant research has been conducted on inverse problems associated with time-independent elliptic operators employing Mittag--Leffler functions or Laplace transform, the study of analogous inverse problems \red{involving an elliptic operator} that varies with time remains notably less developed due to the inapplicability of these tools. 
The unique identification of a time-dependent diffusion coefficient in a one-dimensional model from lateral flux observations was discussed in \cite{Zhang:2016}.  In \cite{ZhangZhou:2023}, 
the backward problem with a time-dependent elliptic operator was examined in case of sufficiently small or large terminal time. Fujishiro and Kian \cite{FujishiroKian:2016} examined the same inverse problem discussed in this paper, establishing a similar stability estimate as shown in Theorem \ref{thm:stab}. 
However, no reconstruction algorithm or numerical analysis exists.
In the current work, we revisit the problem and demonstrate stability through a different approach, utilizing the smoothing properties of the solution operator and choosing an appropriate weighted $L^p$ norm. This stability analysis also inspires an iterative algorithm as well as its convergence analysis. Moreover, an error estimate for the numerical discretization is provided. This is achieved by using the proposed stability analysis and appropriate error estimates that are optimal with respect to the regularity of the problem data for the direct problem. In \cite{JinShinZhou:2023,MaSun:2023}, the authors investigated the stable recovery of a time-dependent potential from the integral measurement $\int_{\Omega} u(x,t)\, \mathrm{d} x$ for all $t \in [0,T]$. The problem addressed in the current paper poses greater challenges. Here, the observation $u(x_0,t)$ demands higher regularity, and the term $\Delta u(x_0, t)$, which cannot be computed directly, introduces additional difficulties in numerical analysis. 
 
The remainder of the paper is structured as follows.
Section \ref{sec:prelim} is dedicated to gathering foundational results related to the forward problem, including well-posedness and regularity estimates. A  Lipschitz-type stability of the IPP is established in Section \ref{sec:stability}. In Section \ref{sec:error_estimate}, we develop an iterative algorithm, accompanied by an exhaustive error analysis for the numerical reconstruction. Finally, Section \ref{sec:numerics} offers numerical experiments that demonstrate the efficacy of our numerical approach. Throughout, we use  $c$ to represent a generic constant whose value may vary with each instance of use, but remains independent of variables such as the noise level $\delta$, the spatial mesh size $h$, the temporal step size $\tau$, the iteration number $k$, and so forth.

\section{Preliminaries}\label{sec:prelim}
In this section, we will present foundational results concerning the solution operators, their smoothing properties, and the well-posedness of the problem. These results will be extensively utilized in the subsequent sections.

Let $A = -\Delta + c_0 \mathrm{I}$ with homogeneous Neumann boundary condition, where the domain is defined by $D(A):=\{v\in L^2(\Omega): -\Delta v  + c_0 v \in L^2(\Omega), \partial_\nu v|_{\partial\Omega}=0\}$ and a fixed constant $c_0 > 0$. Let $\{\la_\ell\}_{\ell=1}^\infty$ and $\{\varphi_\ell\}_{\ell=1}^\infty$ be eigenvalues and eigenfunctions of $A$, respectively. Here we denote the eigenvalues $\{\la_\ell\}_{\ell=1}^\infty$ ordered nondecreasingly with multiplicity counted and the corresponding \red{orthonormal eigenfunctions $\{\varphi_\ell\}_{\ell=1}^\infty$ in $L^2(\Omega)$.}
Then we define 
\begin{equation*}
    A^s v=\sum_{\ell=1}^{\infty}\la_\ell^s (v,\varphi_\ell)\varphi_\ell,~~s\ge 0
\end{equation*}
with its domain $D(A^s) = \{v\in L^2\II:\, A^s v\in L^2\II\}$. 

In complex plane $\mathbb{C}$, we define the sector  $\Sigma_{\theta}:=\{0\neq z\in\mathbb{C}: {\rm arg}(z)\leq\theta\}$ with some fixed $\theta\in(\pi/2,\pi)$.
Then the elliptic operator $A$ satisfies the  resolvent estimate \cite[Theorem 3.7.11]{Arendt:book}: for any $z \in \Sigma_{\theta}$
\begin{equation} \label{eqn:resol}
  \| (z +A)^{-1} \|_{L^2\II \rightarrow L^2\II}\le c_\theta |z|^{-1}.   
\end{equation}
Now we introduce the solution operator
\begin{equation}\label{eqn:sol_op}
    E(t):=\frac{1}{2\pi {\rm i}}\int_{\Gamma_{\theta,\kappa}}   (z^\alpha +A)^{-1}e^{zt}\, \dz \quad\mbox{and}\quad
 F(t):=\frac{1}{2\pi {\rm i}}\int_{\Gamma_{\theta,\kappa }}z^{\alpha-1} 
  (z^\alpha+A)^{-1}e^{zt}\, \dz \red{,}
\end{equation}
with the integration over a contour $\Gamma_{\theta,\kappa}$ in $\Sigma_{\theta}$, defined by
\begin{equation*}
  \Gamma_{\theta,\kappa}=\left\{z\in \mathbb{C}: |z|=\kappa, |\arg z|\le \theta\right\}\cup
  \{z\in \mathbb{C}: z=r e^{\pm\mathrm{i}\theta}, r\ge \kappa\} .
\end{equation*}
Note that the time-fractional diffusion problem \eqref{eqn:fde} could be written \red{in the abstract form}
\begin{equation}\label{eqn:fde-re}
 \partial_t^\alpha u(t) + Au(t) = (c_0 - \rho(t))u(t) + f \quad \forall~ t\in(0,T], \quad \text{with}~~ u(0) = u_0.
\end{equation} 
Then the solution to the direct problem \eqref{eqn:fde} could be represented as \cite[Section 6.2]{Jin:2021}
\begin{equation}\label{eqn:sol-rep}
\begin{aligned}
    u(t)=&F(t)u_0+\int_0^t E(t-s) (f-(\rho(s) - c_0)u(s))\ds.
\end{aligned}
\end{equation} 
In case that $f$ is independent of time, then we apply the  identity \red{$F'(t) = -AE(t)$} to derive
\begin{equation}\label{eqn:sol-rep-1}
\begin{aligned}
    u(t)= \red{F(t)\left(u_0-A^{-1}f\right)+A^{-1}f}-\int_0^t (\rho (s) - c_0)E(t-s) u(s) \ds,
\end{aligned}
\end{equation} 
Next, we present some smoothing properties of the solution operators, which will be used throughout the paper.
\begin{lemma}\label{lem:op} 
Let $F(t)$ and $E(t)$ \red{be the operators} defined in \eqref{eqn:sol_op}. Then, for any $s\in [0,1]$,  the following estimate holds
\begin{equation*}
    t^{s\alpha}\|A^{s}F(t)v\|_{L^2\II}+t^{1-(1-s)\alpha}\|A^{s}E(t)v\|_{L^2\II}\le c\|v\|_{L^2\II},\quad \text{for all} ~t>0,
\end{equation*}
where the positive constant $c$ is independent of $t$.
\end{lemma}
\begin{proof}
When $s=0,1$, the estimates can be found in \cite[Theorem 6.4]{Jin:2021} using the resolvent estimate \eqref{eqn:resol}. When $0<s<1$, the result follows from standard interpolation theory \cite[Proposition 2.3]{Lions:1972}.
\end{proof}

The subsequent lemma details the well-posedness of the solution \red{to problem \eqref{eqn:fde}}. Comparable results were previously established in \cite[Theorem 2.1]{JinShinZhou:2023} through a fixed-point argument based on the smoothing properties delineated in Lemma \ref{lem:op}. Due to the similarity of the arguments, we do not repeat the proof here.

\begin{lemma}\label{thm:sol-reg}
Assume $\rho$ belongs to $\mathcal{B}$ and is  piecewise $C^1$, $u_0 \in D(A^{1+\gamma/2})$, and $f \in D(A^{\gamma/2})$, for a certain $\gamma$ where $\frac{d}{2}<\gamma<2$. Under these conditions, the problem \eqref{eqn:fde} uniquely determines a solution $u \in C^\alpha([0,T];D(A^{\gamma/2})) \cap C([0,T];D(A^{1+\gamma/2}))$. Furthermore, the fractional time derivative $\partial_t^\alpha u \in C([0,T];D(A^{\gamma/2}))$.
\end{lemma}

\section{Analysis of \IPP}\label{sec:stability}
Next, we prove a Lipshitz-type stability for the  IPP  under some mild condition. The stability result will further motivate us to design an iterative algorithm for the numerical inversion. 
We begin by proposing the following assumption.

\begin{assumption}\label{assum:stab}
We suppose that the following conditions hold valid.
\begin{itemize}
    \item [(i)] $\rho\in  \mathcal{B}$ and is   piecewise $C^1$, $u_0 \in D(A^{1+\gamma/2})$, $f\in D(A^{\gamma/2})$, for some $\frac{d}{2}<\gamma<2$. 
    \item [(ii)]  For $x_0\in \overline\Omega$,  there exist positive constants $\underline{c}_u$ and $\overline{c}_u$ such that 
  $\underline{c}_u\le u(x_0,t)\le \overline{c}_u$ for all $t\in[0,T]$.
\end{itemize}
\end{assumption}
Note Assumption \ref{assum:stab}(i) ensures the regularity result in Lemma \ref{thm:sol-reg}. Moreover, by maximum principle \cite[Theroem 12, Section 7.1]{Evans:1998book} Assumption \ref{assum:stab}(ii) holds true if $u_0,g,f$ are strictly positive and bounded. 

We now introduce some notation for norms of $L^p$ space. For $\omega \ge 0$, we define weighted $L^p$ norms by
\begin{equation}\label{eqn:weighted_norm}
 \|  u  \|_{L_\omega^p(0,T)} = \left\{\begin{aligned}
		& \Big(\int_0^T |e^{-\omega t} u(t)|^p \,\dt\Big)^{\frac1p},&& \text{with} ~~ p\in[0,\infty),\\
		&  \text{esssup}_{t\in(0,T)} e^{-\omega t}u(t) ,&& \text{with} ~~p = \infty.
	\end{aligned}\right.
\end{equation}
Note that for any $p \in [1,\infty]$, the standard $L^p(0,T)$ norm
is equivalent to $L_\omega^p(0,T)$ norm. In the context of a Banach space $X$, $L^p(0,T;X)$ denotes the Bochner space, $\| \cdot \|_{L^p(0,T;X)}$
and $\| \cdot \|_{L_\omega^p(0,T;X)}$ denotes the standard and weighted $L^p$ norm, respectively.

\begin{theorem}\label{thm:stab}
Assuming $\rho_1$, $\rho_2$, $u_0$ and $f$ fulfill the conditions specified in Assumption \ref{assum:stab}, let $u_1 = u(\rho_1)$ and $u_2 = u(\rho_2)$ represent the solutions to \eqref{eqn:fde} corresponding to the potentials $\rho_1$ and $\rho_2$, respectively. Then, for a given point $x_0 \in \overline{\Omega}$ and $p \in [1, \infty]$, \red{the following} stability estimate is valid:
\begin{equation*}
    \|\rho_1-\rho_2\|_{L^p(0,T)}\le c\|\partial_t^\alpha (u_1-u_2)(x_0,\cdot)\|_{L^p(0,T)}.
\end{equation*}
Here $c$ \red{depends on} $\underline{c}_u$ and $\overline{c}_u$ in Assumption \ref{assum:stab}(ii), $\alpha$, $\overline{c}_\rho$ ,  $\|f\|_{ D(A^{\gamma/2})}$, $\| A u_1(x_0,\cdot)\|_{C([0,T])}$ and $\|\partial_t^\alpha u_1(x_0,\cdot)\|_{C([0,T])}$. 
\end{theorem}
\begin{proof}
By Assumption \ref{assum:stab}(ii), $u_i(x_0,t)\ge \underline{c}_u>0$.
According to the equation in \eqref{eqn:fde}, $\rho_i$  can be written as
\begin{equation*}
\begin{aligned}
    \rho_i(t)&=\frac{f(x_0)+\Delta u_i(x_0,t)-\partial_t^{\al}u_i(x_0,t)}{u_i(x_0,t)}\\
    &=\frac{f(x_0 )-A u_i(x_0,t)-\partial_t^{\al}u_i(x_0,t) + c_0 u_i(x_0,t)}{u_i(x_0,t)}
\end{aligned}
\end{equation*}
As a result, we observe
\begin{align*}
    (\rho_1 -\rho_2)(t)=&\left(\frac{f(x_0 )}{u_1(x_0,t)}-\frac{f(x_0 )}{u_2(x_0,t)}\right)+\left( \frac{A u_2}{u_2}- \frac{A u_1}{u_1}\right)(x_0,t)+\left(\frac{\partial_t^{\al}u_2 }{u_2 }-\frac{\partial_t^{\al}u_1 }{u_1 }\right)(x_0,t)\\
    =&\mathrm{I}_1+\mathrm{I}_2+\mathrm{I}_3.
\end{align*}
In the following, we establish bounds for $\mathrm{I}_1$, $\mathrm{I}_2$, $\mathrm{I}_3$ separately. By Assumption \ref{assum:stab}, elliptic regularity estimate
and Sobolev embedding theory \cite[Theorem 7.26]{Gilbarg:1983}, we have that $ |f(x_0)|\le \|A^{\gamma/2} f\|_{L^2\II}\le c$, for $\gamma>\frac{d}2$. 
As a result, We arrive at
\begin{equation*}
    \|\mathrm{I}_1\|_{L_{\omega}^p(0,T)}\le \|f\|_{ D(A^{\gamma/2})} \left\|\frac{u_1-u_2 }{u_1 u_2 } (x_0,\cdot)\right\|_{L^p_{\omega}(0,T)} \le c\|(u_1 -u_2)(x_0,\cdot)\|_{L^p_\omega (0,T)}.
\end{equation*}
For $\mathrm{I}_2$, according to Assumption \ref{assum:stab}, Lemma \ref{thm:sol-reg}, the elliptic regularity and Sobolev embedding
theorem,
we have for $\gamma>\frac{d}2$
$$\|A u_1(x_0,\cdot)\|_{C([0,T])}
\le\|A^{1+\frac\gamma2} u_1 \|_{C([0,T];L^2\II)} \le c.$$ 
Then we conclude   
\begin{equation*}
\begin{split}
    \|\mathrm{I}_2\|_{L^p_\omega(0,T)}
    &\le c\|A(u_1-u_2)(x_0,\cdot)\|_{L^p_\omega (0,T)}.
\end{split}
\end{equation*}
Now we establish an estimate for 
$ \|A(u_1-u_2)(x_0,\cdot)\|_{L^p_\omega (0,T)}$.
Denote $w=u_1-u_2$, $w$ satisfies 
 \begin{equation}\label{eqn:difference_eq}
     \partial_t^\alpha w(t) + A w(t) =  [(\rho_2 -\rho_1 )u_1](t)+[(\rho_2 -c_0)(u_2 -u_1)](t),\quad \forall~t\in(0,T].
 \end{equation}
 and $ w(0)=0 $.
According to \eqref{eqn:sol-rep}, we derive that 
\begin{equation*} 
 \begin{aligned}
     \|A^{1+\frac{\gamma}{2}} (u_1-u_2)(t)\|_{L^2\II}  &\le \int_{0}^{t}|(\rho_2-\rho_1)(s)|\,\| E(t-s )A^{1+\frac{\gamma}{2}}  u_1(s)\|_{L^2\II}\ds\\
     &\quad +\int_{0}^{t}|\rho_2(s)-c_0|\,\| E(t-s )A^{1+\frac{\gamma}{2}}  (u_1-u_2)(s)\|_{L^2\II}\ds \\
     &:=\mathrm{I}_{2,1}(t)+\mathrm{I}_{2,2}(t).
 \end{aligned}
\end{equation*}
Since $\|A^{1+\frac{\gamma}{2}} u_1\|_{C([0,T];L^2\II)}\le c$, the Young's inequality for convolution and Lemma \ref{lem:op} lead to
\begin{align*}
    \int_0^T\left(e^{-\omega t} \mathrm{I}_{2,1}(t) \right)^p \dt=&\int_0^T\left( \int_0^t e^{-\omega t}|(\rho_2-\rho_1)(s)|\,\|   E(t-s) A^{1+\frac{\gamma}{2}} u_1(s)\|_{L^2\II}    \ds \right)^p  \dt\\
    \le & c\int_0^T\left( \int_0^t e^{-\omega (t-s)}(t-s)^{ \alpha-1} e^{-\omega s}|\rho_2(s)-\rho_1(s)|    \ds \right)^p  \dt\\
    \le & c\left(\int_0^T    e^{-\omega t}t^{ \alpha-1}   \dt\right)^p\int_0^T  \left(  e^{-\omega t}  |\rho_2(t)-\rho_1(t)| \right)^p \dt\\
    =&c\left(\int_0^T    e^{-\omega t}t^{ \alpha-1}   \dt\right)^p  \|\rho_1-\rho_2\|_{L^p_\omega(0,T)}^p
    \le c (\Gamma( \alpha) \omega^{- \alpha})^p
    \|\rho_1-\rho_2\|_{L^p_\omega(0,T)}^p.
\end{align*}
Thus we conclude that 
\begin{equation*}
    \|\mathrm{I}_{2,1}\|_{L^p_\omega(0,T)}\le c \omega^{- \alpha} \|\rho_1-\rho_2\|_{L^p_\omega(0,T)}. 
\end{equation*}
Since $\|\rho_2\|_{C(0,T)}\le \bar c_\rho$, the term $\mathrm{I}_{2,2}$ can be estimated via a similar argument. Accordingly, we arrive at 
\begin{equation*}
    \|\mathrm{I}_{2,2}\|_{L^p_\omega (0,T)} \le c \omega^{- \alpha} \|A^{1+\frac{\gamma}{2}}  (u_1-u_2)\|_{L^p_\omega(0,T;L^2\II)}.
\end{equation*}
With the estimations for $\mathrm{I}_{2,1}$ and $\mathrm{I}_{2,2}$, taking $\omega$ sufficiently large, we obtain 
\begin{align*}
    \|A^{1+\frac{\gamma}{2}} (u_1-u_2) \|_{L^p_\omega (0,T;L^2\II)}\le  c \omega^{- \alpha} \|\rho_1-\rho_2\|_{L^p_\omega(0,T)},
\end{align*}
and hence
\begin{align*}
      \|\mathrm{I}_2\|_{L^p_\omega(0,T)}
      \le  c\|A^{1+\frac{\gamma}{2}}(u_1-u_2) \|_{L^p_\omega (0,T;L^2\II)}
      \le   c \omega^{-\alpha} \|\rho_1-\rho_2\|_{L^p_\omega(0,T)}.
\end{align*}
For $\mathrm{I}_3$, since Assumption \ref{assum:stab} holds and {$\| \partial_t^\al  u_1(x_0,\cdot)\|_{C[0,T]}\le c$}, we obtain 
\begin{equation*}
    \|\mathrm{I}_3\|_{L^p_\omega(0,T)}
    \le c(\|\partial_t^\alpha(u_1-u_2)(x_0,\cdot)\|_{L^p_\omega (0,T)} + \|(u_1 -u_2)(x_0,\cdot)\|_{L^p_\omega (0,T)}).
\end{equation*}
 
In conclusion, the estimations for $\mathrm{I}_1$, $\mathrm{I}_2$, $\mathrm{I}_3$ gives that
\begin{align*}
    \|\rho_1-\rho_2\|_{L^p_\omega(0,T) }\le  c\Big(\|\partial_t^\alpha(u_1-u_2)(x_0,\cdot)\|_{L^p_\omega (0,T)} + \|(u_1 -u_2)(x_0,\cdot)\|_{L^p_\omega (0,T)}\Big)+c \omega^{- \alpha} \|\rho_1-\rho_2\|_{L^p_\omega(0,T)} .
\end{align*}
Again, we take $\omega$ sufficiently large and get
\begin{equation*}
\|\rho_1-\rho_2\|_{L^p_\omega(0,T) }\le  
c\Big(\|\partial_t^\alpha(u_1-u_2)(x_0,\cdot)\|_{L^p_\omega (0,T)} + \|(u_1 -u_2)(x_0,\cdot)\|_{L^p_\omega (0,T)} \Big).
\end{equation*}
By the norm equivalence, we obtain
\begin{equation}\label{eqn:stab-00}
\| \rho_1 - \rho_2  \|_{L^p(0,T)} \le 
c\Big( \|\partial_t^\alpha(u_1-u_2)(x_0,\cdot)\|_{L^p  (0,T)} +
\|(u_1 -u_2)(x_0,\cdot)\|_{L^p (0,T)} \Big).
\end{equation}
Moreover, 
by  \cite[Theorem 2.13(ii), p. 45]{Jin:2021}, we can write
$w(t)= w(0) + \Gamma(\alpha)^{-1}\int_0^t(t-s)^{\alpha-1} \partial_s^\alpha w(s) \,\ds$.
We further apply Young's inequality to obtain 
\begin{equation*}
\begin{split}
\|(u_1 -u_2)(x_0,\cdot)\|_{L^p  (0,T)} \le c\,\|\partial_t^\alpha(u_1-u_2)(x_0,\cdot)\|_{L^p  (0,T)}.
\end{split}
\end{equation*}
Together with \eqref{eqn:stab-00}, we obtain the stability result presented in the theorem.
\end{proof}

Theorem \ref{thm:stab} provides a Lipschitz-type stability. We observe that the IPP experiences an $\alpha$th order derivative loss, which directly implies that the ill-posedness of IPP intensifies with an increase in the fractional order 
$\alpha$. The findings herein align with those presented in \cite[Theorem 1.2]{FujishiroKian:2016}. This stability estimate not only ensures the potential for stable numerical reconstruction but also plays a crucial role in devising an iterative reconstruction algorithm.

The Lipschitz-type stability  lays the foundation for the creation of a reconstruction algorithm that can retrieve the potential $\rho(t)$ from the observation $u(x_0,t)$, accompanied by error estimates. We will now introduce a straightforward iterative algorithm and demonstrate its convergence within the $L^p(0,T)$ norm. Here we define the following cut-off function 
\begin{equation*}
P_{\mathcal{B}} [a] = \min(\max(a,0),\overline{c}_\rho).
\end{equation*}
Then for $\rho^\dag \in \mathcal{B}$ and any $\rho\in C[0,T]$, we note 
\begin{equation*}
 | P_{\mathcal{B}} [\rho(t)] - \rho^\dag(t) | 
 \le |  \rho(t) - \rho^\dag(t)  |,\quad \forall \, t\in[0,T].
\end{equation*}

\begin{proposition}\label{prop:iter}
Let $\rho^{\dag} \in \mathcal{B}$ and that $u_0$, $g$, and $f$ adhere to Assumption \ref{assum:stab}. Let $u^{\dag} = u(\rho^{\dag})$ be the solution to \eqref{eqn:fde} associated with the potential $\rho^\dag$. Take $x_0$ as any fixed point in the domain $\Omega$ and define $g(t) = u(x_0, t; \rho^{\dag})$ as the measurement data at that point. Starting with any initial guess $\rho_0$ from the admissible set $\mathcal{B}$, we then proceed with the following iterative scheme
\begin{equation}\label{eqn:iter}
    \rho_{k+1}(t) = P_{\mathcal{B}}\Big[\frac{ f(x_0)+\Delta u(x_0,t;\rho_{k}) - \partial^{\alpha}_t  g(t)    }{g(t) }\Big],\quad \forall \, t\in[0,T].
\end{equation}
Then for sufficiently large $\omega$, the sequence of functions $\{\rho_k\}_{k=0}^\infty$ converges to $\rho^\dag$ in $L_\omega^p(0,T)$ and there holds
\begin{equation}\label{eqn:conv}
  \| \rho^\dag - \rho_{k}   \|_{L_\omega^p(0,T)} \le  (c\omega^{- \alpha})^{k} \| \rho^\dag  - \rho_0 \|_{L_\omega^p(0,T)}, \quad k=1,2,\ldots.
\end{equation}
Here  $c$  relies on $\alpha$,  $\overline{c}_\rho$ and $\|A^{1+\frac\gamma2} u^{\dag}\|_{C(0,T;L^2\II)}$, and $\underline{c}_u$, $\overline{c}_u$ given in Assumption \ref{assum:stab}(ii). 
\end{proposition}
\begin{proof}
First, we define $M:\mathcal{B}\rightarrow \mathcal{B}$ s.t. for any $\rho\in \mathcal{B}$,
    \begin{equation*}
        M \rho = P_{\mathcal{B}}\Big[\frac{ f(x_0)+\Delta u(x_0,t;\rho ) - \partial^{\alpha}_t  g(t)    }{g(t) }\Big].
    \end{equation*}
By Lemma \ref{thm:sol-reg}, we conclude $M \rho \in \mathcal{B}$. 
Note that $\rho^{\dag}$ is a fixed point of the operator  $M$. Then we have 
\begin{align*}
|\rho_{k+1}(t)-\rho^{\dag}(t)|=&\left|P_{\mathcal{B}}\Big[\frac{ f(x_0)+\Delta u(x_0,t;\rho_k ) - \partial^{\alpha}_t  g(t)    }{g(t) }\Big] - \rho^{\dag}  \right|\\
\le & \left|  \frac{ f(x_0)+\Delta u(x_0,t;\rho_k ) - \partial^{\alpha}_t  g(t)    }{g(t) }  - \rho^{\dag}  \right|\\
= & \left|  \frac{ A u(x_0,t;\rho_k ) -   A u(x_0,t;\rho^\dag )    }{g(t) }   - c_0 \frac{   u(x_0,t;\rho_k ) -    u(x_0,t;\rho^\dag )    }{g(t) }   \right|,
\end{align*}
where in the second line, we use the stability for any 
$\rho \in \mathcal{B}$ and $ t\in [0,T]$
$$|P_{\mathcal{B}}[\rho(t)]-\rho^{\dag}(t) |\le | \rho(t)-\rho^{\dag}(t) |. $$
Recall that $g(t)=u(x_0,t;\rho^{\dag})$ satisfying Assumption \ref{assum:stab} (ii). This together with the elliptic regularity and Sobolev embedding implies that
\begin{equation*}
|\rho^{\dag}(t)- \rho_{k+1}(t)|\le c \| A^{1+\frac{\gamma}{2}}  ( u(t;\rho^\dag )  - u(t;\rho_k )  )  \|_{L^2\II}
\end{equation*}
for $\gamma>\frac{d}{2}$.
The argument in Theorem \ref{thm:stab} further leads to
\begin{equation*}
\| A^{1+\frac{\gamma}{2}} ( u(t;\rho^\dag ) - u(t;\rho_k )  )  \|_{L^p_{\omega}(0,T;L^2\II) }\le c\omega^{-\alpha}\|\rho_k-\rho^{\dag}\|_{L^p_\omega (0,T)}.
\end{equation*}
Then we conclude 
\begin{equation*}
 \|\rho_{k+1}-\rho^{\dag}\|_{L^p_\omega(0,T)}
 \le c\omega^{- \alpha} \|\rho_k-\rho^{\dag}\|_{L^p_\omega (0,T)},
\end{equation*}
Finally, choosing $\omega$ large enough, the sequence 
$\{\rho^k\}_{k=0}^{\infty}$ will converge to $\rho^{\dag}$ as in \eqref{eqn:conv}.
\end{proof}
\red{\begin{remark}\label{rmk:Dirichlet_case}
    The above stability estimate Theorem \ref{thm:stab} and the reconstruction scheme in Proposition \ref{prop:iter} can be directly extended to the problem with Dirichlet boundary condition. Consider the initial-boundary value problem:
    \begin{equation*}
        \left\{\begin{aligned}
            \partial_t^\alpha u  -  \Delta u  + \rho(t)u &=f , &&\mbox{in } \Omega\times{(0,T]},\\
             u &=0,&&\mbox{on } \partial\Omega\times{(0,T]},\\
            u(0)&=u_0 ,&&\mbox{in }\Omega.
        \end{aligned}\right.
    \end{equation*}
    The inverse problem aims to reconstruct the potential function $\rho(t)$ from the single point measurement $u(x_0,t)$ with $t\in [0,T]$, where $x_0$ is a point in the interior of $\Omega$. Under the Assumption \ref{assum:stab}, one can show that the inverse problem achieves the Lipschitz stability. The unknown potential $\rho(t)$ could be reconstructed by the fixed point iteration formula (\ref{eqn:iter}). We note that in the Dirichlet boundary condition case, the point measurement should be taken for $x_0$ in the interior of the domain, while for the Neumann boundary condition case, the point measurement could be taken at $x_0$ on $\partial\Omega$.
\end{remark}}

\section{Reconstruction algorithm and error analysis} \label{sec:error_estimate}
Recall the iteration \eqref{eqn:iter} in Proposition \ref{prop:iter} gives a approach to reconstruct the unknown potential. In practice, one needs to discretize the forward problem \eqref{eqn:fde} and the iteration scheme \eqref{eqn:iter}. To derive the error estimate, throughout this section, we view the operator $A=-\Delta+c_0 $ as an operator in $L^\infty\II$, with domain 
\begin{equation*}
    D_{\infty}(A)=\{v\in C(\overline{\Omega}), \, \partial_\nu v=0 \mbox{ on }\partial\Omega, \mbox{ and }-\Delta u+c_0u\in C(\overline{\Omega}) \}.
\end{equation*}
Then there holds the maximum norm resolvent estimate (\cite[Theorem 1]{Stewart:1980} and \cite[Appendix A]{LiMa:sinum2022})
\begin{equation}\label{eqn:Linfty_resolvent}
    \|(z+A)^{-1}v\|_{L^\infty\II}\le c \|v\|_{L^\infty\II}.
\end{equation}
This yields the following smoothing property of the solution operators $E(t)$ and $F(t)$ in $L^\infty\II$-norm. The proof follows from the standard argument by Laplace transform and resolvent estimate \eqref{eqn:Linfty_resolvent}. See e.g., \cite[Theorem 6.4]{Jin:2021}.
\begin{lemma}\label{lem:op_Linfty}
The following estimate holds for any $s\in[0,1]$,
\begin{equation*}
    t^{s\alpha}\|A^{s}F(t)v\|_{L^\infty\II}+t^{1-(1-s)\alpha}\|A^{s}E(t)v\|_{L^\infty\II}\le c\|v\|_{L^\infty\II},\quad \text{for all }t >0,
\end{equation*}
where the positive constant $c$ is independent of $t$.
\end{lemma}
 
For the numerical estimate, we propose following higher regularity assumption.
\begin{assumption}\label{assum:err}
We assume following conditions hold
\begin{itemize}
    \item [(i)] $\rho \in  \mathcal{B}$ and is  piecewise $C^1 $ , $u_0 , Au_0 \in D_{\infty}(A)$, $f\in D_{\infty}(A)$. 
    \item [(ii)]  For any $x_0\in \overline\Omega$, there exist positive constant $\underline{c}_u$ and $\overline{c}_u$ such that  $\underline{c}_u\le u(x_0,t)\le \overline{c}_u$ for all $t\in [0,T]$.
\end{itemize}
\end{assumption}
Under Assumption \ref{assum:err}(i), the same argument for Lemma \ref{thm:sol-reg} yield that  problem \eqref{eqn:fde} uniquely determines a solution $u\in C^\al([0,T];D_{\infty}(A))$ and $Au\in C([0,T];D_{\infty}(A))$. Moreover, for a point $x_0\in \overline\Omega$, $ u(x_0,t)\in C^\al[0,T]\cap C^2(0,T] $, and there holds \cite[Theorem 2.1]{WangZhou:2020}
\begin{equation}\label{eqn:al-deriv}
\|\partial_t^{\ell}(\partial_t^{\al}u( t))\|_{L^\infty(\Omega)} +
\|\partial_t^{\ell}(Au( t))\|_{L^\infty(\Omega)}\le c t^{-\ell}     \quad \text{for}~~ \ell=0,1,
\end{equation}
and 
\begin{equation}\label{eqn:deriv}
\|\partial_t^\ell u( t)\|_{L^\infty(\Omega)}\le c t^{\al-\ell} \quad \text{for}~~ \ell=1,2.
\end{equation}

\subsection{Fully discrete scheme for the direct problem}\label{sec:fully}
Next we present fully discrete scheme for solving the forward problem. For time discretization, we divide the interval $[0,T]$ into $N$ uniformly subintervals with step size $\tau=T/N$ and set the time grids $\{t_n=n\tau\}_{n=0}^N$. We employ convolution quadrature generated by backward Euler scheme  (BECQ) \cite[Chapter 3]{JinZhou:2023book} to approximate the fractional derivative $\partial_t^{\al}v(t_n)$ (with $v^j=v(t_j)$):
\begin{equation*}
    \overline{\partial}_{\tau}^{\alpha} v^n=\tau^{-\alpha}\sum_{j=0}^{n} \omega_j^{(\alpha)}(v^{n-j}-v^0),\qquad \text{with } (1-\xi)^{\alpha}=\sum_{j=0}^{n} \omega_j^{(\alpha)}.
\end{equation*}

For spatial discretization, we apply the Galerkin finite element method, following \cite[Chapter 2]{JinZhou:2023book}. Let $\mathcal{T}_h$ be a quasi-uniform simplicial triangulation of domain $\Omega$ with mesh size $h$. Over  $\mathcal{T}_h$, we let $V_h\subset H^1\II$ be the conforming piecewise linear finite element space.
On the FEM space $V_h$, define the orthogonal projection on $L^2\II$, $P_h:L^2\II\rightarrow V_h$ such that
\begin{equation*}
    (P_hv,\phi_h)=(v,\phi_h),\quad \forall\, v\in L^2\II,\, \phi_h\in V_h,
\end{equation*}
and 
Ritz-projection $R_h:H^1\II\rightarrow V_h$ by 
\begin{equation*}
    (\nabla R_hv,\nabla\phi_h)=(\nabla v,\nabla \phi_h)\quad \text{and} \quad \int_{\Omega}R_h v\dx =\int_\Omega v \dx,\quad \forall v\in H^1\II,\, \phi_h\in V_h.
\end{equation*}
Then for $1\le p\le \infty$, $s=0,1,2$ and $k=0,1$  with $k\le s$, the  following estimates holds for $ L^2\II$-projection \cite{Bakaev:2001,Crouzeix:1987}
\begin{equation}\label{eqn:error_l2}
    \|v-P_hv\|_{W^{k,p}\II} \le ch^{s-k}\|v\|_{W^{s,p}\II} \quad \forall\, v\in W^{s,p}\II
\end{equation}
and Ritz projection \cite[eq. (1.45)]{Thomee:2006}
\begin{equation}\label{eqn:error_ritz}
    \|v-R_hv\|_{L^\infty\II} \le ch^s  |\log h|\|v\|_{W^{s,\infty}\II} \quad \forall \, v\in W^{s,\infty}\II.
\end{equation}  
We define the discrete Laplacian operator $ \Delta_h:V_h\rightarrow V_h$ such that
\begin{equation*}
    (-\Delta_h v_h,\phi_h)=(\nabla v_h,\nabla \phi_h) ,\quad \forall v_h,\phi_h\in V_h,
\end{equation*}
and define $A_h=-\Delta_h+c_0$.  Then the following discrete $L^{\infty}\II$ resolvent estimate holds \cite[Theorem 1.1]{CrouzeixLarssonThomee:1994} and \cite[Theorem 1.1]{LiSun:2017}
\begin{equation}\label{eqn:dis_Linfty_resolvent}
     \|(z+A_h)^{-1}v_h\|_{L^\infty\II }\le c|z|^{-1}\|v_h\|_{L^\infty\II}\quad \forall v_h\in V_h\mbox{, }z\in \Sigma_{\theta },\, \theta\in ( {\pi}/{2},\pi).
 \end{equation}

We write the numerical scheme for approximately solving \eqref{eqn:fde}:   find $u_h^n := u_h^n(\rho)\in V_h$ for $n=1,\dots,N$, such that
\begin{equation}\label{eqn:fully-FP}
    (\dtau u_h^n ,\phi_h)+(\nabla u_h^n,\nabla \phi_h)+\rho(t_n)(u_h^n ,\phi_h)=(f,\phi_h).
\end{equation}
with $u_h^0=R_h u_0$. 
Then the scheme \eqref{eqn:fully-FP} could be written as
\begin{equation}\label{eqn:fully-FP-op}
    \dtau u_h^n  + A_h u_h^n = (c_0-\rho(t_n))u_h^n +P_h f.
\end{equation}
Use discrete Laplace transform, we obtain the following  representation of $u_h^n $ 
\begin{equation}\label{eqn:sol-rep-fully}
    u_h^n =F_{h,\tau}^n R_h u_0+ \tau \sum_{j=1}^{n} E_{h,\tau}^{n-j} (P_h f-\rho(t_j)u_h^j(\rho)+c_0 u_h^j(\rho) ),
\end{equation}
where the discrete solution operator $F_{h,\tau}^n $ and $E_{h,\tau}^{n} $ are defined respectively by  \cite[Section 3.2]{JinZhou:2023book}
\begin{equation}\label{eqn:sol_op_fd}
    \begin{aligned}
            F_{h,\tau}^n&=\frac{1}{2\pi {\rm i}}\int_{\Gamma_{\theta,\kappa }^{\tau}}\delta_\tau(e^{-z\tau})^{\al-1}   e^{-z\tau }  e^{zt_n}(\delta_\tau(e^{-z\tau})^\alpha+A_h)^{-1}\, \dz,\\
            E_{h,\tau}^n&=\frac{1}{2\pi {\rm i}}\int_{\Gamma_{\theta,\kappa }^{\tau}}e^{zt_n}   (\delta_\tau(e^{-z\tau})^\alpha+A_h)^{-1}\, \dz,
    \end{aligned}
\end{equation}
 with the kernel function $\delta_\tau(\xi)=\frac{1-\xi}{\tau}$ and the contour $\Gamma_{\theta,\kappa}^{\tau}:=\{z\in\Gamma_{\theta,\kappa}: |\Im(z)|\le \pi/\tau \}$ with $\theta\in (\frac{\pi}{2},\pi)$ close to $\pi/2$ (oriented counterclockwise). Similar as in continuous case, the following smoothing property of the  discrete solution operators $F_{h,\tau}^n $ and $E_{h,\tau}^n$ holds valid 
(cf. \cite[Lemma 3.1]{JinZhou:2023book})
 \begin{equation}\label{eqn:dis_sm_prop}
     t_n^{s\alpha}\|A_h^{s}F_{h,\tau}^n v_h\|_{L^\infty\II}+t_{n+1} ^{1-(1-s)\alpha}\|A_h^{s}E_{h,\tau}^n v_h\|_{L^\infty\II}\le c\|v_h\|_{L^\infty\II} 
     ~\forall v_h\in V_h, ~s\in[0,1].
 \end{equation}
 
In the following, we aim to analyze the  scheme \eqref{eqn:fully-FP}. First of all, we present \textsl{a priori} bounds for the numerical solution $u_h^n$.
\begin{lemma}\label{lem:uhn-apriori}
Assume that $\rho\in \mathcal{B}$, $u_0,Au_0\in D_{\infty}(A)$, and $f \in  D_{\infty}(A)$. Let $u_h^n$ solves the fully discrete problem \eqref{eqn:fully-FP}. Then
there holds for all $s\in[0,1]$
\begin{equation*}
   \max_{1\le n\le N} \| (A_h)^s u_h^n \|_{L^\infty\II} \le C.
\end{equation*} 
\end{lemma}
\begin{proof}
Using the solution representation \eqref{eqn:sol-rep-fully}, we have
\begin{equation}\label{eqn:sol-rep-42}
\begin{split}
    A_h u_h^n &= A_h F_{h,\tau}^n R_h u_0+ \tau \sum_{j=1}^{n} E_{h,\tau}^{n-j}  A_h  P_h f + \tau \sum_{j=1}^{n} E_{h,\tau}^{n-j}    (c_0  - \rho(t_j))A_h u_h^j.
\end{split}
\end{equation} 
\red{Then we apply the identity $ A_h R_h=P_h A+c_0(R_h-P_h)$ and use the smoothing property \eqref{eqn:dis_sm_prop} together with the approximation properties \eqref{eqn:error_l2}-\eqref{eqn:error_ritz} to obtain} 
\begin{equation*}
\| A_h F_{h,\tau}^n R_h u_0 \|_{L^\infty\II} \le c \| A u_0 \|_{L^\infty\II}.
\end{equation*} 
For the second term in \eqref{eqn:sol-rep-42}, we apply the equality that $I- F_{h,\tau}^n = \tau \sum_{j=1}^{n} E_{h,\tau}^{n-j} A_h$ and 
\red{the smoothing property} \eqref{eqn:dis_sm_prop} to obtain
\begin{equation*}
\Big\| \tau \sum_{j=1}^{n} E_{h,\tau}^{n-j}  A_h  P_h f \Big\|_{L^\infty\II} 
\le c \| (I- F_{h,\tau}^n) P_h f \|_{L^\infty\II} \le c  \|  P_h f \|_{L^\infty\II} \le c  \|   f \|_{L^\infty\II}.
\end{equation*} 
For the third term in  \eqref{eqn:sol-rep-42}, we \red{apply the smoothing property} \eqref{eqn:dis_sm_prop} to derive
\begin{equation*}
\Big\| \tau \sum_{j=1}^{n} E_{h,\tau}^{n-j}   
(c_0  - \rho(t_j))A_h u_h^j(q) \|_{L^\infty\II}
\le \tau \sum_{j=1}^{n}  t_{n-j+1}^{\alpha-1}  
\| A_h u_h^j  \|_{L^\infty\II}.
\end{equation*} 
To sum up, we arrive at
\begin{equation*}
 \| A_h u_h^n \|_{L^\infty\II} \le c  (\| A u_0 \|_{L^\infty\II} +\| f \|_{L^\infty\II}) + \tau \sum_{j=1}^{n}  t_{n-j+1}^{\alpha-1}  
\| A_h u_h^j  \|_{L^\infty\II}.
\end{equation*} 
Then the desired result for $s=1$ can be directly derived from \red{Gr\"onwall's inequality \cite[Theorem 10.2]{JinZhou:2023book}}. The estimate for $s=0$ could be obtained analogously and the intermediate cases can be proved by the interpolation technique.
\end{proof}

Now we introduce the problem with time-independent coefficient: given $v(0)=v_0$, find $v(t)\in \text{D}_\infty(A) $ such that 
\begin{equation}\label{eqn:pde_q0}
     \partial_t^\al v(t) + Av(t)  = f(t),\quad \forall ~ t\in (0,T].
\end{equation}
The fully discrete scheme of \eqref{eqn:pde_q0} reads: for $n=1,2,\dots,N$, we look for $v_h^n
\in V_h$ satisfying
\begin{equation}\label{eqn:pde_q0_fd}
 \dtau v_h^n  + A_h v_h^n = P_h f^n  \quad \forall t\in(0,T],\quad \text{with } v_h^0=R_h v_0.
\end{equation}

The error analysis for the numerical scheme \eqref{eqn:pde_q0_fd} has been provided in \cite[Theorem 4.5]{WangZhou:2020}. In the next lemma, we establish an bound for $\|Av(t_n)- A_h v_h^n\|_{L^\infty(\Omega)}$.
\begin{lemma}\label{lem:pde_q0_error}
Suppose that $v_0,Av_0\in D_{\infty}(A)$, and $f \in C(0,T;D_{\infty}(A))\cap W^{1,1}(0,T;D_{\infty}(A))$.
Let $v$ \red{solve problems} \eqref{eqn:pde_q0} while $v_h^n$ solves the numerical scheme  \eqref{eqn:pde_q0_fd}. Then there holds
\begin{equation*}
   \|Av(t_n)- A_h v_h^n\|_{L^\infty(\Omega)} \leq c(h^2|\log h|^3+\tau t_n^{-1}).
\end{equation*}
\end{lemma}
\begin{proof}
To derive the error estimate, we design an auxillary function $v_h$ solving the semidiscrete problem: find $v_h(t)\in V_h$ such that 
\begin{equation}\label{eqn:pde_q0_sd}
\partial_t^\al v_h(t) + A_h v_h(t) = P_h f(t) \quad \forall  \, t\in (0,T],\quad \text{with}~~v_h(0)=R_h v_0.
\end{equation}
Similar to \eqref{eqn:sol-rep}, $v_h$ can be represented by \cite[Section 2.3]{JinZhou:2023book}
\begin{align}\label{eqn:sol-rep-h}
    v_h(t)= F_h(t)R_h v_0+\int_0^t E_h(t-s) P_h f(s)\ds,
\end{align}
where 
\begin{equation*} 
    E_h(t):=\frac{1}{2\pi {\rm i}}\int_{\Gamma_{\theta,\kappa}} (z^\alpha +A_h)^{-1}e^{zt} \, \dz \quad\mbox{and}\quad
F_h(t):=\frac{1}{2\pi {\rm i}}\int_{\Gamma_{\theta,\kappa }}z^{\alpha-1}  (z^\alpha+A_h)^{-1}e^{zt} \, \dz 
\end{equation*}
We first estimate $\| Av(t)-A_hv_h(t)\|_{L^\infty\II} $. \red{With the solution representation}, we have
\begin{align*}
    Av(t)-A_hv_h(t)= & \left(F(t)A v_0-F_h(t)A_h R_h v_0\right) + \int_0^t E_h(t-s) A_h(R_h-P_h)f \ds \\
      &+ \int_0^t E(t-s)Af -E_h(t-s)A_hR_hf \ds=\mathrm{I}_1(t)+\mathrm{I}_2(t)+\mathrm{I}_3(t).
\end{align*}
To estimate $\mathrm{I}_1$, we note $A_hR_h=P_hA+c_0(R_h-P_h)$ and obtain
\begin{align*}
    \mathrm{I}_1(t)=  (F(t) - F_h(t) P_h) A v_0 + c_0 F_h(t) (R_h-P_h) v_0.
\end{align*}
The estimates \eqref{eqn:dis_sm_prop}, \eqref{eqn:error_l2} and \eqref{eqn:error_ritz} lead to
\begin{align*}
  \|F_h(t) (R_h-P_h) v_0\|_{L^\infty\II} \le c  \|  (R_h-P_h) v_0\|_{L^\infty\II} \le c h^2  |\log h| \| v_0 \|_{W^{2,\infty}\II}.
\end{align*}
Meanwhile, according to \cite[Lemma 4.2]{WangZhou:2020}, we derive that
\begin{align*}
  \|( F_h(t) P_h-F(t) ) A v_0\|_{L^\infty\II}\le c h^2  |\log h|^3 \| A v_0 \|_{W^{2,\infty}\II}.
\end{align*}
Therefore, we obtain
$$ \|  \mathrm{I}_1(t) \|_{L^\infty\II} \le c h^2  |\log h|^3.$$
To estimate $\mathrm{I}_2$, since \red{$-A_hE_h(t)=\frac{d}{dt}F_h(t)$}, we write 
\begin{align*}
    \mathrm{I}_2(t)=&\int_0^t E_h(s) A_h(R_h-P_h)f(t-s) \ds\\
                =&\red{-\int_0^t F_h(s)  (R_h-P_h)f'(t-s) \ds-F_h(t)(R_h-P_h)f(0)+(R_h-P_h)f(t)}.
\end{align*}
The estimate \eqref{eqn:dis_sm_prop}, the projection error estimates \eqref{eqn:error_l2} and \eqref{eqn:error_ritz}  directly imply that
\begin{equation*}
    \|\mathrm{I}_2(t)\|_{L^\infty\II}\le ch^2 |\log h| \left(\int_0^t \|f'(s)\|_{W^{2,\infty}\II}\ds+\|f\|_{C([0,T];W^{2,\infty}\II)}\right).
\end{equation*}
To estimate $\mathrm{I}_3$, still we apply  \red{$-AE(t)=\frac{d}{dt}F(t)$} and write
\begin{align*}
    \red{\mathrm{I}_3(t)=\int_0^t (F_h(s)R_h-F(s) )f'(t-s) \ds+(F_h(t)R_h-F(t))f(0)+(I-R_h)f(t).}
\end{align*}
Using the estimate \cite[Lemma 4.2]{WangZhou:2020} and the approximation error of Ritz projection in \eqref{eqn:error_ritz} lead to
\begin{align*}
      \|\mathrm{I}_3(t)\|_{L^\infty\II}\le ch^2 |\log h|^3 \left(\int_0^t \|f'(s)\|_{W^{2,\infty}\II}\ds+\|f\|_{C([0,T];W^{2,\infty}\II)}\right).
\end{align*}

Now we consider $\|  A_hv_h(t_n)- A_h v_h^n\|_{L^\infty\II} $. By the solution representation \eqref{eqn:sol-rep-h} and \eqref{eqn:sol-rep-fully}, we have 
\begin{align*}
     A_hv_h(t_n)- A_h v_h^n
     &=\left(F_h(t_n)-F_{h,\tau}^n\right)A_h R_h v_0+
     \Big[\int_0^t E_h(t-s)A_h P_h f(s)\ds- \tau \sum_{j=1}^{n}E_{h,\tau}^{n-j} A_hP_h f^j \Big]\\
     &=:\mathrm{II}_1^n +\mathrm{II}_2^n.
\end{align*}
With the identity $ A_h R_h=P_h A+c_0(R_h-P_h)$, we can write
\begin{align*}
    \mathrm{II}_1^n = \left(F_h(t_n)-F_{h,\tau}^n\right) P_h A v_0 + c_0 \left(F_h(t_n)-F_{h,\tau}^n\right)(R_h-P_h)v_0.
\end{align*}
Using the error estimate for $ F_h(t_n)-F_{h,\tau}^n $
 \cite[Lemma 4.9]{ZhangZhou:2023} we have
\begin{align*}
\| \mathrm{II}_1^n \|_{L^\infty\II} \le c\tau t_n^{-1} \| v_0 \|_{W^{2,\infty}\II}.
\end{align*}
To estimate $\mathrm{II}_2^n$, we insert $R_h$ and use $A_hR_h=P_hA+c_0(R_h-P_h)$,
\begin{align*}
    \mathrm{II}_2^n=& \left(\int_0^t E_h(t-s)A_h (P_h-R_h) f(s)\ds- \tau \sum_{j=1}^{n}E_{h,\tau}^{n-j} A_h(P_h-R_h) f^j\right)\\
    &+\left(\int_0^t E_h(t-s)P_h A f(s)\ds- \tau \sum_{j=1}^{n}E_{h,\tau}^{n-j} P_h A f^j\right) \\
    &+c_0\left(\int_0^t E_h(t-s) (R_h-P_h) f(s)\ds- \tau \sum_{j=1}^{n}E_{h,\tau}^{n-j}  (R_h-P_h) f^j\right) \\
    &= \mathrm{II}_{2,1}^n+\mathrm{II}_{2,2}^n+\mathrm{II}_{2,3}^n
\end{align*}
The terms $ \mathrm{II}_{2,1} $ and $ \mathrm{II}_{2,3} $ can be estimated with the same argument for $\mathrm{I}_2$
\begin{align*}
    \| \mathrm{II}_{2,1}^n\|_{L^\infty\II} + \| \mathrm{II}_{2,3}^n\|_{L^\infty\II} \le      ch^2 |\log h|\left(\|Af\|_{C([0,T];L^\infty\II)}+\int_0^t \|Af'(s)\|_{L^\infty\II)}\ds\right).
\end{align*}
Using the similar argument in \cite[Theorem 3.4]{JinZhou:2023book}, we achieve
\begin{align*}
    \| \mathrm{II}_{2,2}^n\|_{L^\infty\II}\le      c\tau t_n^{\al-1}\left(\|A f\|_{C([0,T];L^\infty\II)}+\int_0^t \|Af'(s)\|_{L^\infty \II)}\ds\right).
\end{align*}
Combining the bounds of $\mathrm{I}_1,\mathrm{I}_2, \mathrm{I}_3$, $\mathrm{II}_1$ and $\mathrm{II}_2$, the proof is completed.
\end{proof}

The following lemma provides the error analysis for the numerical scheme \eqref{eqn:fully-FP}.
\begin{lemma}\label{lem:pde_error}
Suppose $\rho^\dag\in \mathcal{B}\cap C^1[0,T]$, $u_0,Au_0,f\in D_{\infty}(A)$.
\red{Let $u= u(\rho^\dag)$ solve problem} \eqref{eqn:fde} while $u_h^n = u_h^n(\rho^\dag)$ solves the  scheme \eqref{eqn:fully-FP}. 
Then the following estimates hold:
\begin{equation*}
\begin{aligned}
 \| u(t_n)-  u_h^n\|_{L^\infty(\Omega)} \leq &c( h^2 |\log h|^3+\tau t_n^{\alpha-1}),\\
 \|Au(t_n)- A_h u_h^n\|_{L^\infty(\Omega)} \leq &c(h^2|\log h|^3+ \tau |\log \tau|t_n^{-1}).
\end{aligned}
\end{equation*} 
\end{lemma}
\begin{proof}
The first estimate has been provided in \cite[Lemmas 3.2 and  4.2]{WangZhou:2020}. Then it suffices to show the second assertion. To this end,
we define $\overline{u}_h^n$ satisfying  
\begin{equation}\label{eqn:baru}
 \dtau \overline{u}_h^n  + A_h \overline{u}_h^n = P_h  f +(c_0 - \rho^\dag(t_n) ) P_h u(t_n;\rho^\dag) \quad \text{for}~~  n=1,\dots,N,
\end{equation}
and $\overline{u}_h^0=R_h u_0 $. Then
we can split the error as  
\begin{equation*}
    e^n:=Au(t_n)- A_h u_h^n=(Au(t_n)- A_h\overline{u}_h^n) + (A_h\overline{u}_h^n- A_h u_h^n)=:\vartheta^n+\varrho^n.
\end{equation*}
Using the  solution regularity \eqref{eqn:deriv} and Lemma \ref{lem:pde_q0_error}, we derive
\begin{equation*}
    \|\vartheta^n\|_{L^\infty\II}\le c(h^2 |\log h|^3+\tau t_n^{-1}).
\end{equation*}
To estimate $\varrho^n$, note  $ \overline{u}_h^0-   u_h^0=0  $ and 
\begin{equation*}
\dtau (\overline{u}_h^n-   u_h^n) + A_h   (\overline{u}_h^n-   u_h^n) = (c_0-\rho^\dag(t_n))( u_h^n-P_h u(t_n)).
\end{equation*}
Hence $\varrho^n$ admits following representation 
\begin{align*}
    \varrho^n= & \tau \sum_{j=1}^{n} E_{h,\tau}^{n-j} \left( (c_0-\rho^\dag(t_j))A_h(u_h^j-P_hu(t_j)) \right)\\
    = & \tau \sum_{j=1}^{n} E_{h,\tau}^{n-j} \left( (c_0-\rho^\dag(t_j))(A_h u_h^j- A u(t_j)) \right)+\tau \sum_{j=1}^{n} E_{h,\tau}^{n-j} \left( (c_0-\rho^\dag(t_j))(  A u(t_j)-A_hR_h u(t_j)) \right)\\
    &+\tau \sum_{j=1}^{n} E_{h,\tau}^{n-j} \left( (c_0-\rho^\dag(t_j))(   A_hR_h u(t_j))-A_hP_h u(t_j)) \right)=\varrho_1^n+\varrho_1^n+\varrho_3^n
\end{align*}
By {smoothing property of $  E_{h,\tau}^{n-j} $} in \eqref{eqn:dis_sm_prop} and smoothness of $\rho^\dag$, we immediately get
\begin{equation*}
    \| \varrho_1^n\|_{L^\infty\II}\le c \tau \sum_{j=1}^{n}  t_{n-j+1} ^{\al-1}   \| A_h u_h^j -A u(t_j) \|_{L^\infty\II} \le c \tau \sum_{j=1}^{n}  t_{n-j+1}^{\al-1}   \| e_j\|_{L^\infty\II}.
\end{equation*}
Since $A_h R_h=P_h A+c_0(R_h-P_h)$, \red{with the projection error} \eqref{eqn:error_l2}, we derive 
\begin{align*}
    \| \varrho_2^n\|_{L^\infty\II}\le & c \tau \sum_{j=1}^{n}  t_{n-j+1}^{\al-1} \left(  \|   A u(t_j)-P_h A u(t_j) \|_{L^\infty\II} + c_0\|  (R_h-P_h) u(t_j) \|_{L^\infty\II}\right)\\
    \le & c \tau  \sum_{j=1}^{n}  t_{n-j+1}^{\al-1}  \left(h^2 \| Au(t_j)\|_{W^{2,\infty}\II} +h^2|\log h| \| u(t_j)\|_{W^{2,\infty}\II}  \right)\\
    \le & c  t_j^{\al}h^2\| Au(t_j)\|_{W^{2,\infty}\II}+c  t_j^{\al}h^2|\log h| \|  u(t_j)\|_{W^{2,\infty}\II}.
\end{align*}
To estimate $ \varrho_3^n$, with {smoothing property of $E_{h,\tau}^{n-j}$} in \eqref{eqn:dis_sm_prop} and projection error \eqref{eqn:error_l2} and \eqref{eqn:error_ritz}, we obtain
\begin{align*}
\|\varrho_3^n\|_{L^\infty\II}\le   c\tau \sum_{j=1}^{n}   t_{n-j+1} ^{\epsilon\alpha-1}\|A_h^{\epsilon}(R_h-P_h)u(t_j)\|_{L^\infty\II}. 
\end{align*}
Here we apply the inverse inequality for any $\phi_h \in V_h$ and $p\in[1,\infty]$ \cite[Theorem 4.5.11]{Brenner:2008}
\begin{align*}
\| A_h  \phi_h \|_{L^\infty\II} &\le 
c h^{-d/p}\| A_h  \phi_h \|_{L^p\II} 
= c h^{-d/p} \sup_{\psi_h \in V_h} \frac{(\nabla \phi_h, \nabla \psi_h) + c_0(\phi_h,\psi_h)}{\| \psi_h  \|_{L^{p^*}\II}}\\
& \le c h^{-d/p - 2} \| \phi_h \|_{L^p\II}
\le c h^{-d/p - 2} \| \phi_h \|_{L^\infty\II}.
\end{align*}
By interpolation inequality, we conclude that
\begin{align*}
\| A_h^\epsilon  \phi_h \|_{L^\infty\II} 
 \le c h^{-d/p - 2} \| \phi_h \|_{L^p\II}
\le c h^{-(d/p + 2)\epsilon} \| \phi_h \|_{L^\infty\II}.
\end{align*}
Then we obtain the estimate for $\varrho_3^n$ by taking $\epsilon=1/|\log h|$ such that
\begin{equation*}
     \|\varrho_3^n\|_{L^\infty\II}\le  c \epsilon^{-1} h^{2-(d/p + 2)\epsilon} |\log h| \le c h^2  |\log h|^2 .
\end{equation*}
Combining above estimates, we derive 
\begin{align*}
    \| e^n\|_{L^\infty\II}\le c \tau \sum_{j=1}^{n}  t_{n-j+1} ^{\al-1}   \| e_j\|_{L^\infty\II}+c(h^2|\log h|^3+\tau t_n^{-1}).
\end{align*}
Then the desired estimate follows immediately by using discrete Gr\"onwall’s inequality.
\end{proof}

\subsection{Numerical reconstruction for  \IPP}\label{inverse_potential}
In this part, we state the numerical reconstruction scheme for  IPP. Recall that the measurement is taken at a fixed point $x_0\in \overline\Omega$: $g(t)=u^{\dag}(x_0,t)$, for $t\in [0,T]$. Throughout, we assume that we have $C([0,T])$ noisy measurement data $ g_\delta$ satisfying \eqref{eqn:noise}.
Under Assumption \ref{assum:stab} (ii),  $g\ge \underline{c}_u>0$ is strictly positive. Hence we may assume the noisy measurement $g_\delta$ is also strictly positive:
\begin{equation}\label{eqn:positive_data}
    0<\frac{\underline{c}_u}{2}\le g_\delta\le \overline{c}_u+\frac{\underline{c}_u}{2}.
\end{equation}
We define the admissible set for discretized potential as 
\begin{equation*}
    \mathcal{B}_N:=\{\rho=(\rho^n)_{n=1}^N:0\le \rho^n\le \overline{c}_\rho\}.
\end{equation*}
Then the numerical reconstruction scheme for IPP is given as follows: let $\rho_0=(\rho_0^n)_{n=1}^N\in \mathcal{B}_N$ be an initial guess, we update $\rho_{k+1}=(\rho_{k+1}^n)_{n=1}^N $ from $\rho_{k}=(\rho_{k}^n)_{n=1}^N $ by 
\begin{equation}\label{eqn:iter_dis}
    \rho_{k+1}^n=P_{\mathcal{B}}\left[\frac{f(x_0)+\Delta_h u_h^n(x_0;\rho_k)-\dtau g_\delta(t_n)}{g_\delta(t_n)} \right],\quad n=1,\dots,N,
\end{equation}
where $ u_h^n(x_0;\rho_k)$ is the solution of \eqref{eqn:fully-FP} with potential $\rho_k$. 

In the following, we want to show the convergence of the iteration scheme \eqref{eqn:iter_dis} and analyze the reconstruction error. We first introduce the  $\ell^p$ norm and weighted $\ell^p$ norm.
For a sequence $v_n\in X$, $n=1,\dots$, we define  $\ell^p(X)$ norm  as follows
$$
\|(v^n)_{n=1}^\infty\|_{\ell^p(X)}:=
\left\{
\begin{aligned}
&\bigg(\sum_{n=1}^\infty\tau\|v^n\|_{X}^p\bigg)^{\frac{1}{p}},  &&\mbox{if}\,\,\, 1\le p<\infty,\\
&\sup_{n\ge 1}\|v^n\|_{X}, &&\mbox{if}\,\,\, p=\infty .
\end{aligned}\right.
$$
For a sequence $v_n\in X$, $n=1,\dots$, we define the weighted $\ell^p_\omega(X)$ norm with a fixed $\omega\ge 0$ as follows
$$
\|[v^n]_{n=1}^\infty\|_{\ell^p_\omega(X)}:=
\left\{
\begin{aligned}
&\bigg(\tau\sum_{n=1}^\infty (e^{-\omega t_n}\|v^n\|_{X})^p\bigg)^{\frac{1}{p}},  &&\mbox{if}\,\,\, 1\le p<\infty,\\
&\sup_{n\ge 1} e^{-\omega t_n}\|v^n\|_{X}, &&\mbox{if}\,\,\, p=\infty .
\end{aligned}\right.
$$
For any sequence $(v^n)_{n=1}^\infty$ and $p\in[0,\infty]$, it is straightforward to observe that
$\|\cdot\|_{\ell^p_\omega(X)}$ and $\|\cdot\|_{\ell^p(X)}$ are equivalent.
Throughout, if $X=\mathbb{R}$, we simplify the notation of $\ell^p(\mathbb{R})$ and $\ell^p_\omega(\mathbb{R})$ to $\ell^p$ and $\ell_\omega^p$, respectively.
\begin{theorem}\label{thm:recon}
Let Assumption \ref{assum:err} be valid and the observation $g_\delta$ satisfy \eqref{eqn:noise}. 
Then with any  $\rho_0=(\rho_0^n)_{n=1}^N \in \mathcal{B}_N$,
the iteration \eqref{eqn:iter_dis} generates a sequence $\rho_k = (\rho_k^n)_{n=1}^N \in \mathcal{B}_N$, that converges to a limit $\rho_* = (\rho_*^n)_{n=1}^N\in \mathcal{B}_N$  such that
\begin{equation}\label{eqn:conv-dis}
  \| [\rho_{k}^n  -  \rho_{*}^n]_{n=1}^N \|_{\ell_\omega^p} \le  (c \omega^{-\alpha})^{k}\| [\rho_{0}^n  -  \rho_{*}^n]_{n=1}^N \|_{\ell_\omega^p},\quad \text{for all}~~ p\in(1,\infty),
\end{equation}
when the weight parameter $\omega$ is sufficiently large. In addition,  there holds the error estimate
\begin{equation}\label{eqn:inv-err}
  \| [\rho^\dag(t_n) - \rho_*^n  ]_{n=1}^N \|_{\ell^p } \le   c\left(\tau^{1/p}|\log\tau | + h^{2}|\log h|^3 + \tau^{-\al} \delta\right).
\end{equation} 
\end{theorem}
\begin{proof}
First, we show the convergence of the iteration scheme \eqref{eqn:iter_dis} by the contraction mapping theorem. We define $M_{h,\tau}:\mathcal{B}\rightarrow \mathcal{B}_N$ s.t. for any $\rho\in \mathcal{B}_N$
    \begin{equation*}
        (M_{h,\tau}\rho)^n= P_{\mathcal{B}}\left[\frac{f(x_0)+\Delta_h u_h^n(x_0;\rho)-\dtau g_\delta(t_n)}{g_\delta(t_n)} \right],\quad n=1,\dots,N.
    \end{equation*}
In the following, we prove that $M_{h,\tau}$ is a contraction mapping with $\ell_\omega^p$ topology for sufficiently large $\omega$. For $\rho_1,\, \rho_2\in \mathcal{B}_N$, we use the stability of the cut-off operator, positivity of measurement \eqref{eqn:positive_data} and obtain that
    \begin{align*}
        |(M_{h,\tau}\rho_1)^n-(M_{h,\tau}\rho_2)^n|=& \left|\frac{A_h u_h^n(x_0;\rho_1)-A_h u_h^n(x_0;\rho_2)}{g_\delta(t_n)}-c_0\frac{ u_h^n(x_0;\rho_1)- u_h^n(x_0;\rho_2)}{g_\delta(t_n)}\right|\\
        \le& c|A_h u_h^n(x_0;\rho_1)-A_h u_h^n(x_0;\rho_2) | + c |  u_h^n(x_0;\rho_1)-  u_h^n(x_0;\rho_2) |.
    \end{align*}
Denote $w_h^n= u_h^n(x_0;\rho_1)-  u_h^n(x_0;\rho_2)$ which satisfying $w_h^0=0$ and 
    \begin{equation*}
        \dtau w_h^n+A_h w_h^n=(\rho_2(t_n)- \rho_1(t_n))u_h^n(\rho_2) -(\rho_1(t_n)-c_0) w_h^n.
    \end{equation*} 
    We have following solution representation \red{of $w_h^n$}:
    \begin{equation*}
        w_h^n=\tau\sum_{j=1}^N E_{h,\tau}^{n-j} \left((\rho_2^j-\rho_1^j)u_h^j(\rho_2)-(\rho_1^j-c_0) w_h^j \right)
    \end{equation*}
    and hence 
    \begin{equation*}
        A_h w_h^n=\tau\sum_{j=1}^N E_{h,\tau}^{n-j} \left((\rho_2^j-\rho_1^j)A_h u_h^j(\rho_2)-(\rho_1^j-c_0) A_h w_h^j \right).
    \end{equation*}  
Since $0\le \rho_1^j, \rho_2^j\le \overline{c}_\rho  $ and  by Lemma \ref{lem:uhn-apriori} we have $ \| A_h  u_h^j(\rho_2)\|_{L^2\II}\le c$ for all $j=1,\dots,N$. Then the smoothing property of $E_{h,\tau}^j$ \eqref{eqn:dis_sm_prop} implies that
\begin{equation*}
\|A_h  w_h^n \|_{L^\infty\II}\le c\tau \sum_{j=1}^{n} t_{n-j+1} ^{\al-1}\left(|\rho_2^j-\rho_1^j|+\|A_h w_h^j\|_{L^2\II} \right)
\end{equation*}
Taking $\ell^p_\omega$ norm on both sides, the Young's inequality implies that
    \begin{align*}
        \| (A_h  w_h^n)_{n=1}^N \|_{\ell_\omega^p(L^\infty\II)}\le & c\left(\sum_{n=1}^N \tau^{p+1}\left| \sum_{j=1}^n e^{-\omega t_{n-j}} t_{n-j+1} ^{\al-1} e^{-\omega t_j}\left( |\rho_2^j-\rho_1^j|+\|A_h w_h^j\|_{L^\infty\II}\right)  \right|^p \right)^{\frac{1}{p}} \\
        \le & c\left(\tau \sum_{n=1}^N e^{-\omega t_{n-1}} t_n^{\al-1} \right)\left(  \| ( \rho_2^n-\rho_1^n)_{n=1}^N \|_{\ell_\omega^p}+  \| (A_h  w_h^n)_{n=1}^N \|_{\ell_\omega^p(L^\infty\II)}   \right)\\
        \le & c \omega^{-\alpha}\left(  \| ( \rho_2^n-\rho_1^n)_{n=1}^N \|_{\ell_\omega^p}+  \| (A_h  w_h^n)_{n=1}^N \|_{\ell_\omega^p(L^\infty\II)}   \right)
    \end{align*} 
As a consequence, \red{by taking $\omega$ sufficiently large, we obtain}
    \begin{equation}\label{eqn:Delta_du_Linfty}
        \begin{aligned}
            \| (A_h  w_h^n)_{n=1}^N \|_{\ell_\omega^p(L^\infty\II)} \le 
              \| ( \rho_2^n-\rho_1^n)_{n=1}^N \|_{\ell_\omega^p} (c \omega^{-\alpha} ).  
        \end{aligned}
    \end{equation}
Similarly, we may derive 
\begin{equation}\label{eqn:Delta_du_Linfty2}
\| (   w_h^n)_{n=1}^N \|_{\ell_\omega^p(L^\infty\II)} \le    \| ( \rho_2^n-\rho_1^n)_{n=1}^N \|_{\ell_\omega^p} (c \omega^{-\alpha} ).
\end{equation}
Again, \red{by taking $\omega$ sufficiently large}, $M_{h,\tau}$ is a contraction mapping on $\mathcal{B}_N$, i.e.
    \begin{equation*}
         \|  ( (K_{h,\tau}\rho_1)^n-(K_{h,\tau}\rho_2)^n)_{n=1}^N \|_{\ell_\omega^p} \le    \| ( \rho_2^n-\rho_1^n)_{n=1}^N \|_{\ell_\omega^p} (c \omega^{-\alpha} )\quad\forall~ \rho_1, \rho_2 \in \mathcal{B}_N.
    \end{equation*}
Then we conclude that the sequence $\{\rho_k\}$ converges to a limit $\rho_*=(\rho_*^n)_{n=1}^N\in \mathcal{B}_N$ such that
    \begin{equation*}
        \|  ( \rho_k^n-\rho_*^n)_{n=1}^N \|_{\ell_\omega^p} \le  \| ( \rho_0^n-\rho_*^n)_{n=1}^N \|_{\ell_\omega^p}  ( c\omega^{-\alpha})^k  .
    \end{equation*}
    
Next, we study the reconstruction error between the limit $(\rho_*^n)_{n=1}^N \in \mathcal{B}_N$ and the exact potential $\rho^\dag\in \mathcal{B}$. Since $(\rho_*^n)_{n=1}^N$ is a fixed point of $M_{h,\tau}$, by the stability of cut-off operator $P_{\mathcal{B}}$, we have
    \begin{align*}
        |\rho^\dag(t_n)-\rho_*^n|= & \left|\rho^\dag(t_n)- P_{\mathcal{B}}\left[\frac{f(x_0)+\Delta_h u_h^n(x_0;\rho_*)-\dtau g_\delta(t_n)}{g_\delta(t_n)} \right] \right|\\
        \le &\left| \frac{f(x_0)+\Delta     u(x_0,t_n;\rho^\dag)-\partial_t^\al g(t_n)}{g(t_n)} - \frac{f(x_0)+\Delta_h  u_h^n(x_0;\rho_*)-\dtau g_\delta(t_n)}{g_\delta(t_n)}  \right|\\
        \le & \left| \frac{f(x_0)}{g(t_n)}- \frac{f(x_0)}{g_\delta(t_n)}\right|+\left| \frac{A  u(x_0,t_n;\rho^\dag)}{g(t_n)} -\frac{ A_h u_h^n(x_0;\rho_*)}{g_\delta(t_n)}  \right|  +c_0\left|\frac{u(x_0,t_n;\rho^\dag)}{g(t_n)} - \frac{   u_h^n(x_0;\rho_*)}{g_\delta(t_n)}  \right|\\
        &+\left| \frac{\partial_t^\al g(t_n)}{g(t_n)}-\frac{\dtau g_\delta(t_n)}{g_\delta(t_n)}   \right|   = \mathrm{I}_1^n+\mathrm{I}_2^n+\mathrm{I}_3^n+\mathrm{I}_4^n.
    \end{align*}
    \red{Upon recalling} Assumption \ref{assum:err}, we have $|f(x_0)|\le  \|f\|_{L^\infty\II}\le c$. Then the \red{strict positivity} of $g(t)$, $g_\delta(t)$ directly implies that
    \begin{equation*}
        \|(\mathrm{I}_1^n)_{n=1}^N\|_{\ell^p}\le c \|f\|_{L^\infty\II } \|(g_\delta(t_n)-g(t_n))_{n=1}^N\|_{\ell^p}\le c\delta.
    \end{equation*}
   To estimate  $\mathrm{I}_2$, Assumption \ref{assum:err} (ii) and condition \eqref{eqn:positive_data} yield that
    \begin{align*}
        |\mathrm{I}_2^n|\le &\left|\frac{g_\delta(t_n) A u(x_0,t_n;\rho^\dag)-g(t_n)A_h u_h^n(x_0;\rho_*)  }{g_\delta(t_n)g(t_n)} \right|\\
        \le & c\left(| (g_\delta(t_n)-g(t_n))A u(x_0,t_n;\rho^\dag)| + |g(t_n)(A u(x_0,t_n;\rho^\dag)-   A_h u_h^n(x_0;\rho_*)   )|\right)\\
        \le &  c \delta\| A u(x_0,t_n;\rho^\dag)\|_{C([0,T])}+c| A_h u_h^n(x_0;\rho_*)-A u(x_0,t_n;\rho^\dag)|\\
        \le & c\delta
        + c \|A u(t_n;\rho^\dag) -A_h u_h^n(\rho^\dag)\|_{L^\infty\II}
        + c \|A_h u_h^n(\rho^\dag)- A_h u_h^n(\rho_*)\|_{L^\infty\II}.
    \end{align*}
Lemma \ref{lem:pde_error} directly implies that
\begin{equation*}
\|A u(t_n;\rho^\dag) -A_h u_h^n(\rho^\dag)\|_{L^\infty\II}
\le c(h^2|\log h|^3+\tau |\log \tau|t_n^{-1}).
\end{equation*}
For the other term $ \| A_h u_h^n(\rho^\dag)- A_h u_h^n(\rho_*) \|_{L^\infty\II}$, by \eqref{eqn:Delta_du_Linfty}, we obtain 
\begin{equation*}
\| (A_h u_h^n( \rho^\dag)-A_h  u_h^n( \rho_*))_{n=1}^N \|_{\ell_\omega^p(L^\infty\II)} \le  c \omega^{-\alpha}  \| (\rho^{\dag}(t_n)- \rho_*^n)_{n=1}^N \|_{\ell_\omega^p},
\end{equation*}
    for sufficiently large $\omega$.
    Take the $\ell^p_\omega$ norm of $(\mathrm{I}_2^n)_{n=1}^N$, we obtain
    \begin{align*}
         \|(\mathrm{I}_2^n)_{n=1}^N\|_{\ell^p_{\omega}}\le c\left(\delta+h^{2 } |\log h|^3+\tau^{1/p} |\log \tau| \right)+c \omega^{-\alpha}  \| ( \rho^{\dag}(t_n)-\rho_*^n)_{n=1}^N \|_{\ell_\omega^p}
    \end{align*}
    For the term $\mathrm{I}_3$, Assumption \ref{assum:err} (ii) and condition   \eqref{eqn:positive_data} yield that
\begin{align*}
|\mathrm{I}_3^n|
\le & c| (g_\delta(t_n)-g(t_n)) u(x_0,t_n;\rho^\dag)| +c|g(t_n)(    u_h^n(x_0;\rho_*)-u(x_0,t_n;\rho^\dag)   )|\\
\le &  c \delta \|u(t_n;\rho^\dag)\|_{L^\infty\II} + c \|  u( t_n;\rho^\dag) -u_h^n( \rho_*)\|_{L^\infty\II}
\end{align*}
Lemma \ref{lem:pde_error} directly implies that
\begin{equation*}
|   u(x_0,t_n;\rho^\dag) - u_h^n(x_0;\rho^\dag)|\le c(h^2 |\log h|^3 +\tau t_n^{\al-1})
\end{equation*}
Take the $\ell^p_\omega$ norm of $(\mathrm{I}_3^n)_{n=1}^N$, we obtain
\begin{align*}
         \|(\mathrm{I}_3^n)_{n=1}^N\|_{\ell^p_{\omega}(\mathbb{R})}\le c\left(\delta+h^{2 }|\log h|^3+\tau^{\al+1/p} \right)+c \omega^{-\alpha}  \| ( \rho^{\dag}(t_n)-\rho_*^n)_{n=1}^N \|_{\ell_\omega^p}
\end{align*} 
    Now we consider $\mathrm{I}_4$. By Assumption \ref{assum:err} (ii) and condition \eqref{eqn:positive_data}, we derive
    \begin{align*}
        |\mathrm{I}_4^n|\le &\left|\frac{g_\delta(t_n)\partial_t^\al g(t_n)-g(t_n)\dtau g_\delta(t_n)}{g(t_n)g_\delta(t_n)} \right|\\
        \le & c\Big(|g(t_n)(   \partial_t^\al g(t_n) -\dtau g_\delta(t_n)   )| +  | (g_\delta(t_n)-g(t_n))\partial_t^\al g(t_n)|\Big)\\
        \le & c \Big( |\partial_t^\al g(t_n)-\dtau g_\delta(t_n)  |+\delta\Big).
    \end{align*}
    This and \cite[Lemma 4.3]{JinShinZhou:2023} imply 
$   \|(\mathrm{I}_4^n)_{n=1}^N\|_{\ell^p} \le c(\tau^{1/p}|\log\tau|+\delta   \tau^{-\alpha})$. 
    \red{Combining above estimates} yields 
    \begin{align*}
        \|( \rho^\dag(t_n)-\rho_*^n)_{n=1}^N\|_{\ell^p_\omega(\mathbb{R})}   \le  c\left(h^{2 }|\log h|^3 +\tau^{1/p}|\log\tau |+\tau^{-\al} \delta\right)
        +c \omega^{-\alpha}  \| (\rho^{\dag}(t_n)- \rho_*^n)_{n=1}^N \|_{\ell_\omega^p}.
    \end{align*}
\red{Letting $\omega$ be large enough and applying the equivalence of $\ell^p$ and $\ell^p_\omega$} lead to the desired estimate.
\end{proof}

The error analysis provided in Theorem \ref{thm:recon} offers practical guidance for selecting the discretization parameters $h$ (the spatial mesh size) and $\tau$ (the temporal step size), by properly balancing the terms, i.e. 
$$ \tau^{1/p}|\log \tau| \sim h^2|\log h|^3 \sim \tau^{-\al} \delta.$$ With the choice of discretization parameters $h \sim \delta^{1/(2p\al+2)}$ and $\tau \sim \delta^{p/(p\al+1)}$, we have following almost optimal error estimate
\begin{equation}\label{eqn:optimal}
\|( \rho^\dag(t_n)-\rho_*^n)_{n=1}^N \|_{\ell^p}    \le  c \delta^{1/(p\al+1)}|\log \delta|^3.
\end{equation}
It is important to note that the error estimate \eqref{eqn:inv-err} involves terms $ \tau^{1/p}|\log \tau|$ and $\delta \tau^{-\alpha}$. Therefore, an excessively large or small time step size can result in substantial errors in the numerical recovery.
This phenomenon is clearly demonstrated in the numerical experiments; see Figure \ref{fig:semi-1d} for illustration.


\section{Numerical experiments}\label{sec:numerics}
We now present some experimental results to demonstrate the analysis results. To generate exact measurement data $g(t) = u(x_0,t)$, we first solve the time-fractional partial differential equations (\ref{eqn:fde}) with some specified data $f$ and $u_0$, by employing the Galerkin finite element method for spatial discretization and  convolution quadrature generated by backward Euler for time discretization, as elaborated in Section \ref{sec:fully}, using a refined space-time grid for high precision. Then we add some noise to $g$ to get 
\begin{equation*}
g_\delta(t_n) = g(t_n) + \epsilon \xi(t_n),\quad n=1,\dots,N,
\end{equation*}
where $\{t_n =nT/N\}_{n=0}^N$ are the equally partition points of $[0,T]$, each $\xi(t_n)$ is uniformly distributed in $[-1,1]$ and $\epsilon\geq0$ indicates the relative noise level, i.e., $\epsilon = \max\{g(t_n), n=0,1,2,\cdots,N\}\times\delta/100$. Then, in order to reconstruct $\rho^\dag$ from the noisy measurement data $g_\delta$, we use the iterative algorithm in \eqref{eqn:iter_dis}. The iteration starts from the initial guess $q_0 \equiv 2$. Even though our numerical scheme and its analysis were done for the linear problem where the source term $f$ is simply a function of $x$, we tested our algorithm for both linear and nonlinear source term $f$. It is observed that the algorithm converges within 5 iterations for the linear case and 60 iteration for the nonlinear case.

We focus on a one dimensional linear equation with the domain $\Omega = (0,1)$ and the specified source and initial-boundary conditions $f(x) = 1 + 20x^2(1-x)^2$ and $u_0(x) = 2 + \cos{2\pi x} $. In order to test the applicability of the proposed method to different conditions than the Assumption \ref{assum:stab}, we run experiments on
three different potential functions:
\begin{itemize}
\item[(i)] Smooth potential: $\rho_1^\dag(t) =  \exp(\cos(5t))$.
\item[(ii)] Piecewise smooth and continuous potential
\begin{equation*}
\rho_2^{\dagger}(t) = \left\{\begin{aligned}
		& ~~~~~\tfrac{8}{T}t+0.7,\quad 0\leq t \leq \tfrac{T}{4},\\
		& -\tfrac{8}{T}t +4.7,\quad \tfrac{T}{4}\leq t \leq \tfrac{T}{2},\\
  		&~~~~~\tfrac{8}{T}t -3.3 ,\quad \tfrac{T}{2} \leq t \leq \tfrac{3}{4}T, \\
  		& -\tfrac{8}{T}t+8.7,\quad  \tfrac{3}{4}T\leq t \leq T.
	\end{aligned}\right.
\end{equation*}
\item[(iii)] Discontinuous potential: 
\begin{equation*}
\rho_3^{\dagger}(t) = \left\{\begin{aligned}
		& 1,\quad ~~~0\leq t < \tfrac{T}{4},\\
		& 2.5,\quad \tfrac{T}{4}\leq t < \tfrac{T}{2},\\
  		& 1.5,\quad \tfrac{T}{2} \leq t < \tfrac{3}{4}T, \\
  		& 2,\quad  ~~\tfrac{3}{4}T\leq t \leq T.
	\end{aligned}\right.\quad
\end{equation*}
\end{itemize}
For the figure of merit, we use the $\ell^2$-error $e(\rho_*)= \|(\rho^{\dagger}(t_n)- \rho_*^n)_{n=1}^N\|_{\ell^2}$ for the reconstruction $\rho_*$.
For noisy observational data satisfying \eqref{eqn:noise}, we take the temporal step size 
$\tau \sim \delta^{2/(2\alpha +1)}$ and spatial mesh size $h \sim  \delta^{1/2(2\alpha +1)} $ according to the error estimate  \eqref{eqn:inv-err}. Throughout, we choose $p=2$ for all experiments. 

Firstly, we provide a numerical experiment to examine the sharpness of the error bound \eqref{eqn:inv-err}. We fix $\delta = 0$ and investigate the effect of the predicted discretization error, i.e. $O(h^2|\log h|^3 + \tau^{1/p}|\log \tau |)$. The spatial convergence results are depicted in Fig. \ref{fig:conv-h-1d} where we took $T = 0.5$ and $\tau = T/800$. It demonstrates that the convergence is about $O(h^2)$ for all potentials and we can remark that convergence is irrelevant to the fractional order $\alpha$ which can be expected from the terms in the error analysis, $h^2|\log h|^3 + \tau^{1/p}|\log \tau |$.   
Next, the temporal convergence results are presented in Fig. \ref{fig:conv-t-1d}, where we fix $h=1/100$ and $T=0.5$. It exhibits that the convergence rate is comparable to $O(\tau^{0.5})$ for all potentials.
Fig. \ref{fig:conv-delta-1d} shows the convergence rate with respect to the noise level $\delta$. To this end, for a given $\tau$, we set $h \sim \tau^{0.25}$ and $\delta\sim\tau^{\alpha + 0.5}$, and then we change $\tau$ to adjust the noise level $\delta$ accordingly. In the figure, we can observe an $O(\delta^{0.5})$ empirical convergence rate. 

The reconstructions from noisy and noisy-free data are shown in Fig. \ref{fig:recon-1d} with $\alpha = 0.5$ and the optimal time step size $\tau = T/2^8$ with $T=0.5$. In the case of the exact data, the reconstructions follow closely to the exact potentials $\rho^{\dagger}$ except for those points where the potentials are not continuous or near the initial time $t=0$. For the case $\epsilon =0.1\%$, the numerical reconstructions show minor oscillations. It seems that the reconstruction method is rather unstable at those points of jump discontinuities including the starting time $t=0$ while it is quite stable on those points of continuous sharp corner. Also, Figure \ref{fig:recon-1d}(c) implies that the proposed reconstruction method could be applicable for discontinuous potentials for which the Assumption \ref{assum:err} is violated.

\begin{figure}[hbt!]
\centering
\setlength{\tabcolsep}{0pt}
\begin{tabular}{ccc}
\includegraphics[width=.33\textwidth]{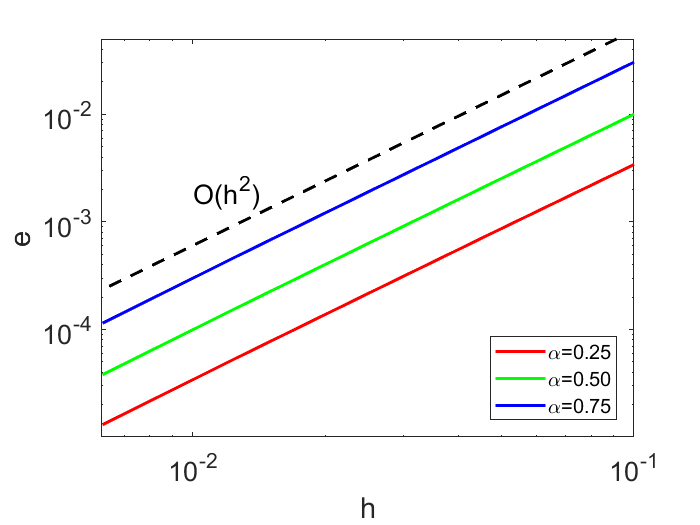} &
\includegraphics[width=.33\textwidth]{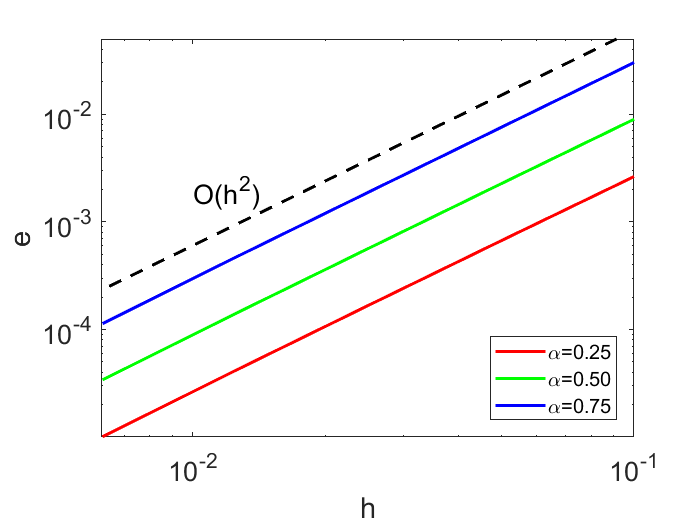} &
\includegraphics[width=.33\textwidth]{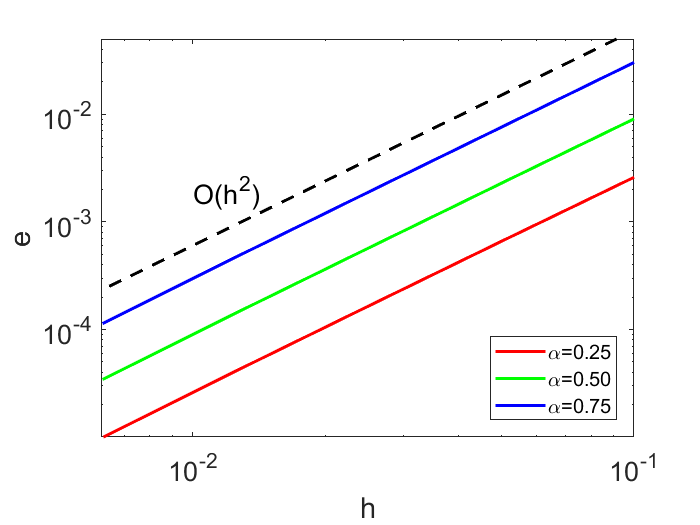} \\
(a) $\rho_1^\dag$ & (b) $\rho_2^\dag$ & (c) $\rho_3^\dag$
\end{tabular}
\caption{The spatial convergence for $\alpha =0.25, 0.5$, and $0.75$, with exact observational data. The black dashed line is the plot for $O(h^2)$ convergence rate. \label{fig:conv-h-1d}}
\end{figure}

\begin{figure}[hbt!]
\centering
\setlength{\tabcolsep}{0pt}
\begin{tabular}{ccc}
\includegraphics[width=.33\textwidth]{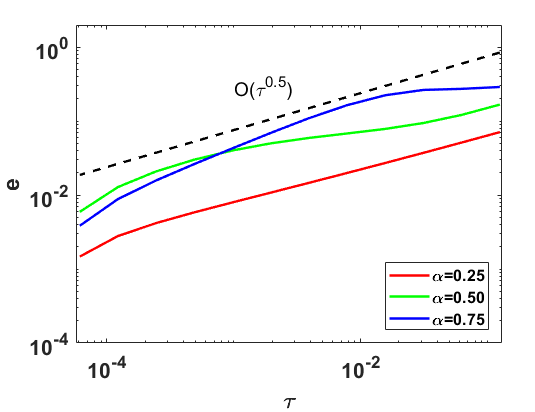} &
\includegraphics[width=.33\textwidth]{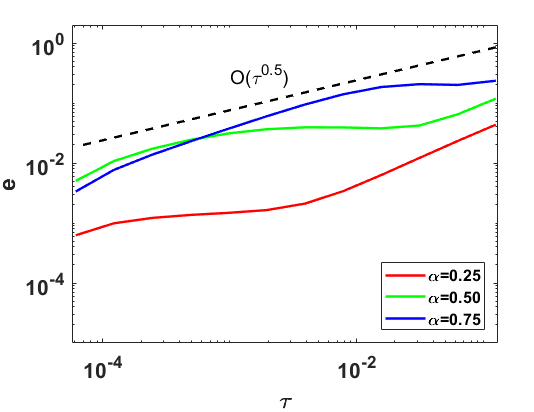} &
\includegraphics[width=.33\textwidth]{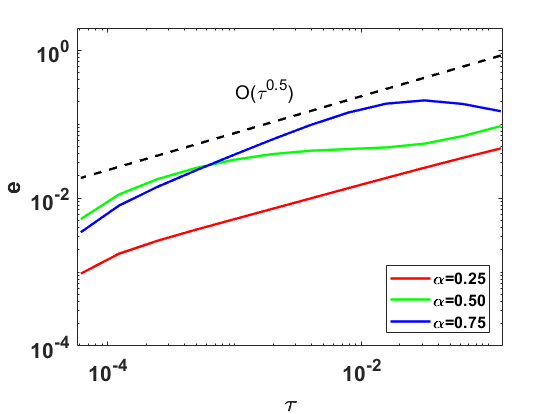} \\
(a) $\rho_1^\dag$ & (b) $\rho_2^\dag$ & (c) $\rho_3^\dag$
\end{tabular}
\caption{The temporal convergence for $\alpha =0.25, 0.5$, and $0.75$, with exact observational data.  The black dashed line is the plot for $O(\tau^{0.5})$ convergence rate. \label{fig:conv-t-1d} }
\end{figure}

\begin{figure}[hbt!]
\centering
\setlength{\tabcolsep}{0pt}
\begin{tabular}{ccc}
\includegraphics[width=.33\textwidth]{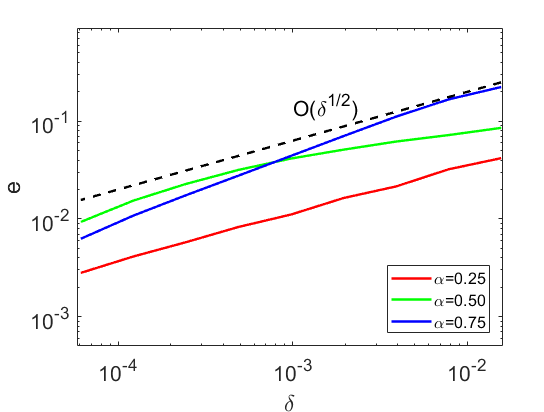} &
\includegraphics[width=.33\textwidth]{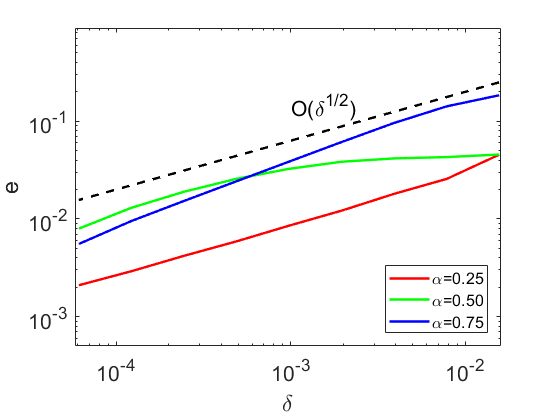} &
\includegraphics[width=.33\textwidth]{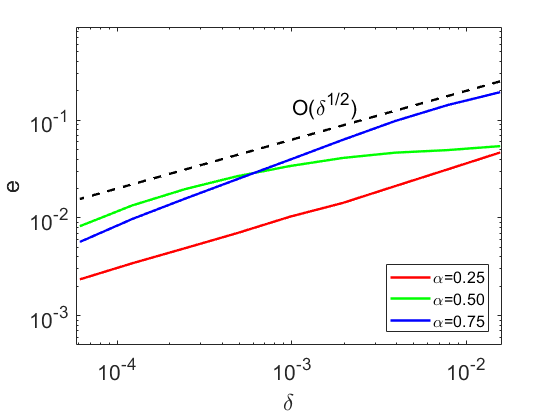} \\
(a) $\rho_1^\dag$ & (b) $\rho_2^\dag$ & (c) $\rho_3^\dag$
\end{tabular}
\caption{The convergence with respect to noise level $\delta$ for $\alpha =0.25, 0.5$ and $0.75$. The approximation is obtained by setting the discretization parameters $h$ and $\tau$ according to \eqref{eqn:optimal}.  The black dashed line is the plot for $O(\delta^{1/2})$ convergence rate.  \label{fig:conv-delta-1d}}
\end{figure}


\begin{figure}
\centering
\setlength{\tabcolsep}{0pt}
\begin{tabular}{ccc}
\includegraphics[width=.33\textwidth]{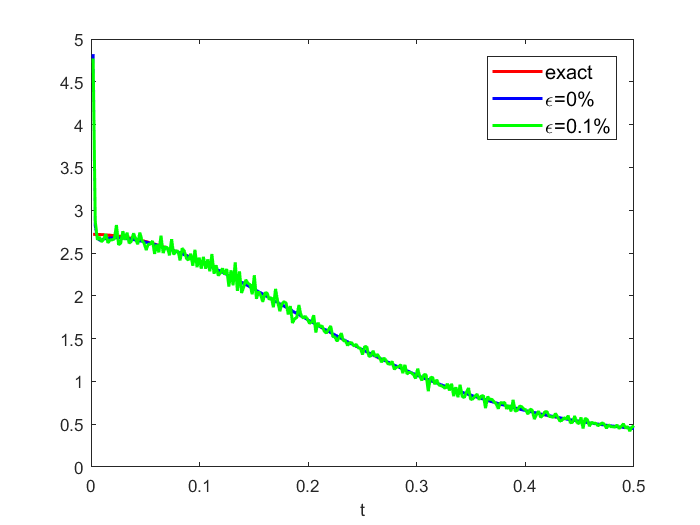} &
\includegraphics[width=.33\textwidth]{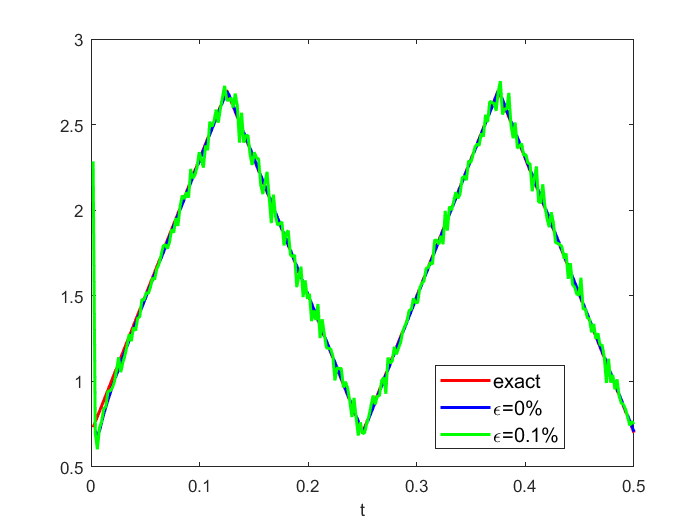} & 
\includegraphics[width=.33\textwidth]{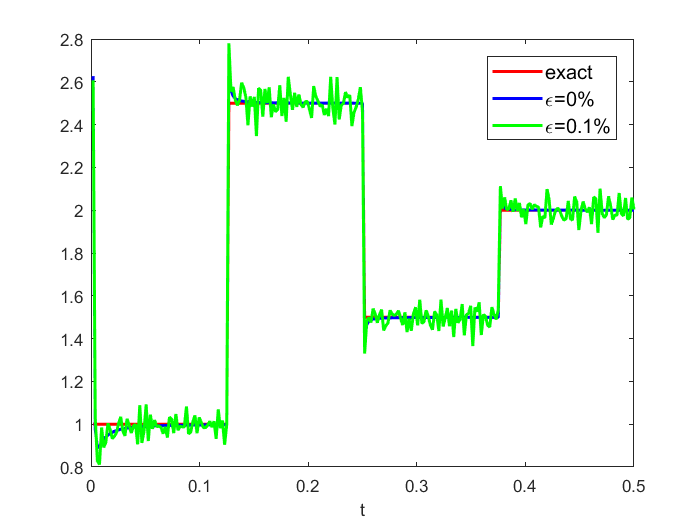}\\
(a) $\rho_1^\dag $ & (b) $\rho_2^\dag$ & (c) $\rho_3^\dag$
\end{tabular}
\caption{The reconstruction results from the noisy and noisy-free data, with $\alpha=0.5$. Taking optimal time step size $\tau=T/2^8$ when $T=0.5$. The lines in red indicate the exact potentials, lines in blue are reconstructions from the noisy-free data, and the green lines are reconstructions from the noisy data.\label{fig:recon-1d}}  
\end{figure}

The theoretical analysis in Theorem \ref{thm:recon} indicates the fractional order $\alpha$ has a significant impact on the convergence of the proposed reconstruction scheme. Indeed, the factor $\omega^{-\alpha}$ in \eqref{eqn:conv-dis} illustrates that the iteration will converge  much faster when we choose larger $\alpha$ or $\omega$.  Since the weighted $\ell^2_{\omega}$-norm is utilized in the theoretical analysis, we study the convergence in $\ell^2_{\omega}$-norm and compare the results with the ones in $\ell^2$-norm.
For the linear problems, we observed that the convergence is so very fast (within 5 iterations) that the comparison between $\ell^2$-norm and $\ell^2_{\omega}$-norm are not so noticeable. Instead, we test the algorithm for a nonlinear problem 
 \begin{equation}\label{eqn:fde-nonlinear}
 \left\{\begin{aligned}
     \partial_t^\alpha u -  \Delta u  + \rho(t)u&=f(u), &&\mbox{in } \Omega\times{(0,T]},\\
      \partial_\nu u&=0,&&\mbox{on } \partial\Omega\times{(0,T]},\\
    u(0)&=u_0,&&\mbox{in }\Omega,
  \end{aligned}\right.
 \end{equation}
where we set the source term by $f(u) = (u-1)(u-3)$. Then the iterative reconstruction scheme for the nonlinear problem is as follows: starting with any initial guess $\rho_0=(\rho_0^n)_{n=1}^N $ in the admissible set $\mathcal{B}$, update $\rho_{k+1}=(\rho_{k+1}^n)_{n=1}^N $   by 
\begin{equation*}\label{eqn:iter_dis_nonlinear}
\rho_{k+1}^n=P_{\mathcal{B}}\left[\frac{f(u^{n-1}_h(x_0;\rho_k))+\Delta_h u_h^n(x_0;\rho_k)-\dtau g_\delta(t_n)}{g_\delta(t_n)} \right],\quad n=1,\dots,N.
\end{equation*}
Fig. \ref{fig:conv-fp} depicts the results with the nonlinear problem. In the figure, plots in the left column are for the exact data and the ones in the right column are for the noisy measurement. The top row exhibits the $\ell^2$ error, and the bottom row presents the  $\ell^2_{\omega}$ error, with $\omega = 10$. In all cases, we set $T = 5$ and $\tau = T/2^{10}$. It is observed that the convergence is much faster in $\ell^2_{\omega}$-norm for both exact data and noisy data, and we can observe the linear convergence consistently for the weighted $\ell^2_{\omega}$-norm.

The error estimate delineated in \eqref{eqn:inv-err} underscores that the regularizing effect for solving inverse problems is largely attributed to the time discretization. Therefore, a judicious selection of the time step size $\tau$ is paramount for the approximation $\rho_*$ to attain optimal accuracy. This regularizing effect is visually depicted in Fig. \ref{fig:semi-1d}, where the reconstruction error diminishes with a decrease in step size $\tau$ up to a point, beyond which it starts to rise again, underscoring the critical nature of optimizing $\tau$. An optimally chosen step size yields reconstructions that are not only accurate but also exhibit minimal oscillations, thereby validating the conditional stability of the inverse problem as stated in Theorem \ref{thm:stab}. Conversely, too large or too small step size would lead to substantial errors in reconstructions—either from the pronounced discretization error, as indicated by the term $\tau^{1/p}|\log \tau|$, or from the amplified effect of noise, as represented by the term $\tau^{-\alpha}\delta$.

\begin{figure}[hbt!]
\centering
\setlength{\tabcolsep}{0pt}
\begin{tabular}{cc}
\includegraphics[width=.38\textwidth]{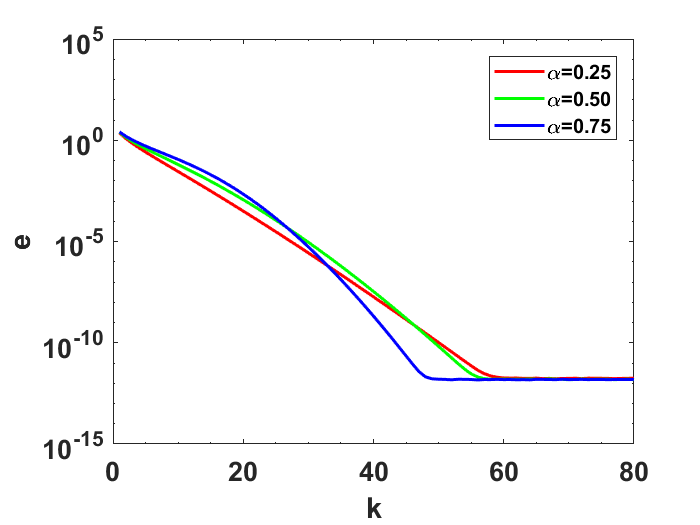} &
\includegraphics[width=.38\textwidth]{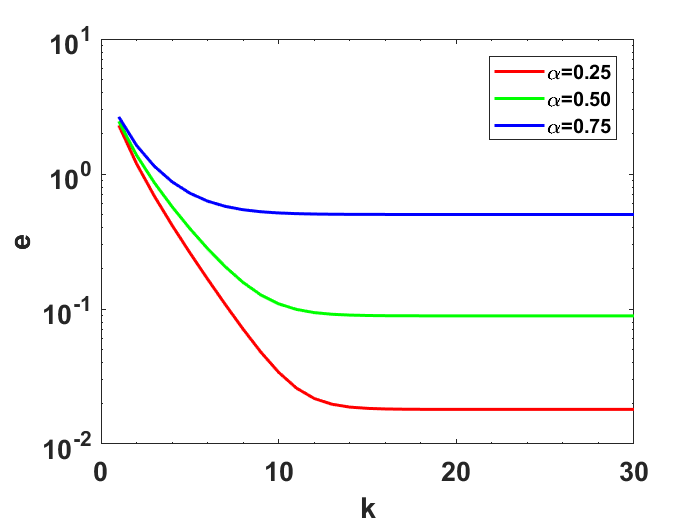} \\
\includegraphics[width=.38\textwidth]{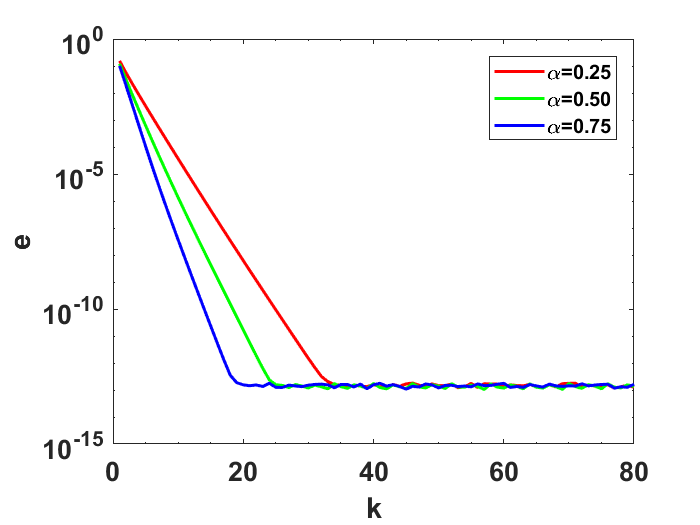} &
\includegraphics[width=.38\textwidth]{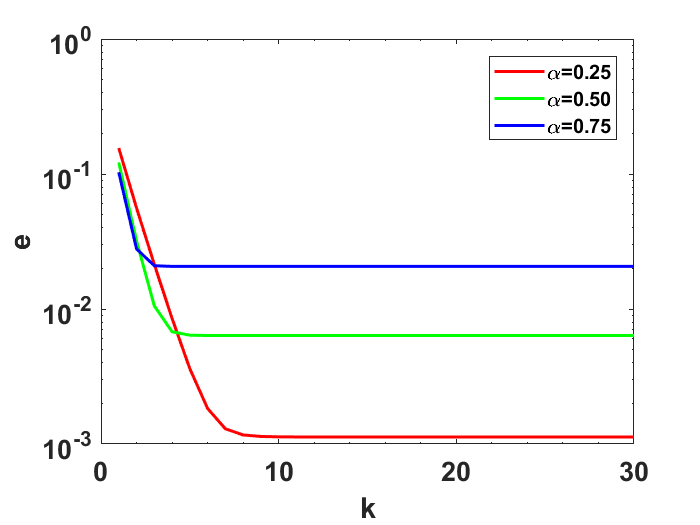} \\
(a) $\epsilon = 0\%$ & (b) $\epsilon = 1\%$
\end{tabular}
\caption{The decay of error throughout the iterations. $k$ denotes the number of iterations. The first row: errors in $\ell^p$. The second row: errors in $\ell^p_\omega $ with $\omega=10$. The first column: errors from the exact data. The second column: errors from the noisy data.
\label{fig:conv-fp}}
\end{figure}

\begin{figure}[hbt!]
\centering
\setlength{\tabcolsep}{0pt}
\begin{tabular}{ccc}
\includegraphics[width=.33\textwidth]{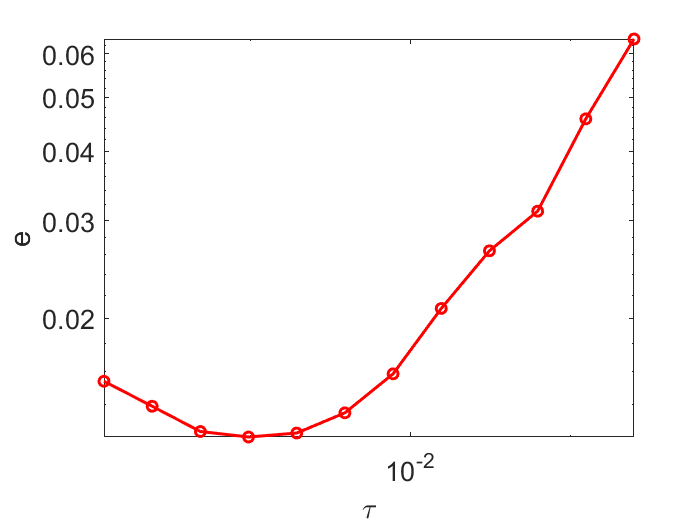} &
\includegraphics[width=.33\textwidth]{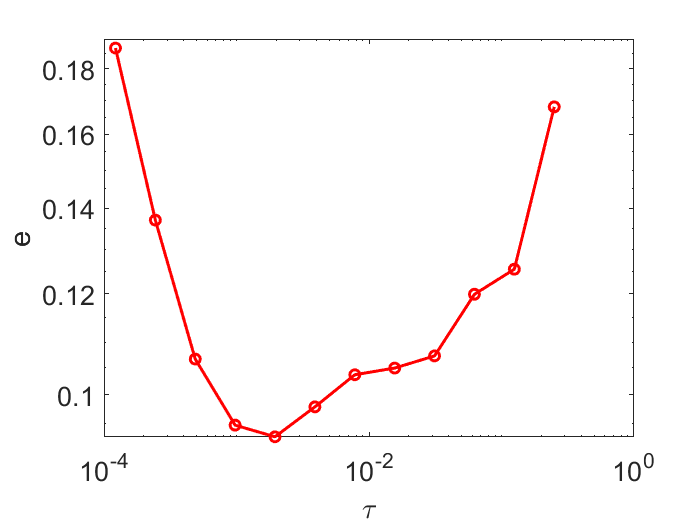} & 
\includegraphics[width=.33\textwidth]{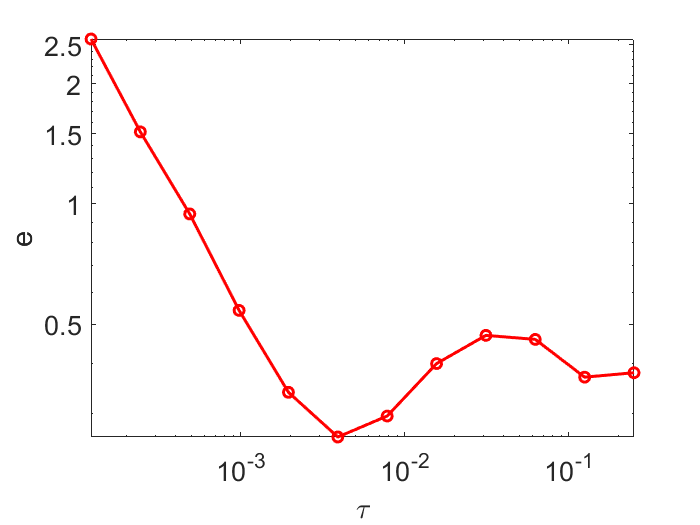}\\
\includegraphics[width=.33\textwidth]{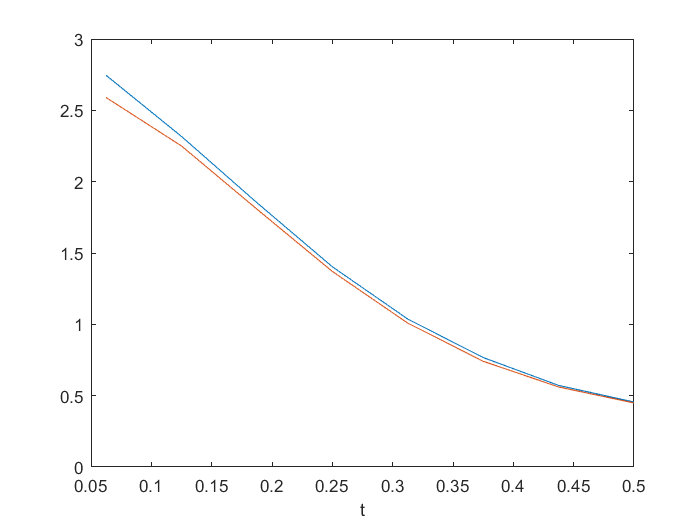} &
\includegraphics[width=.33\textwidth]{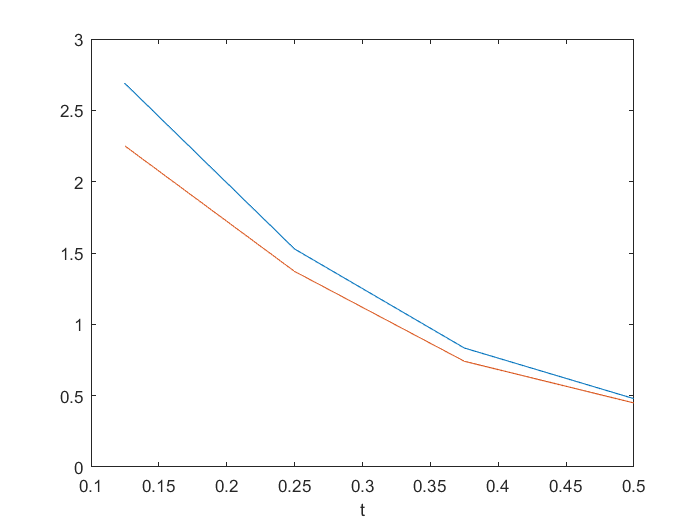} &
\includegraphics[width=.33\textwidth]{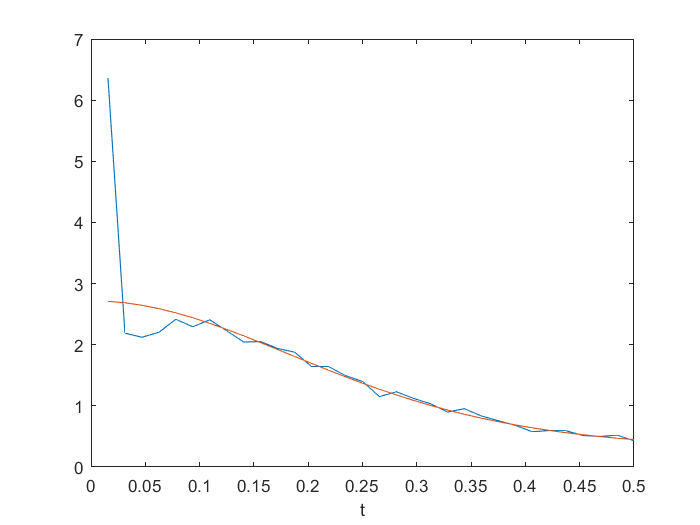} \\
\includegraphics[width=.33\textwidth]{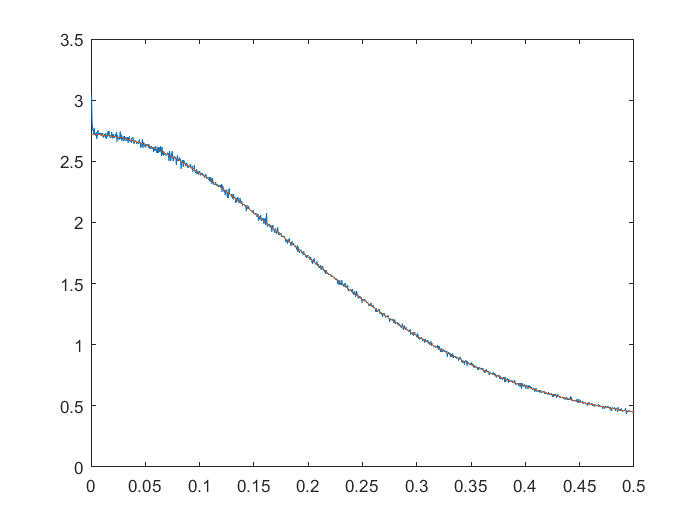} &
\includegraphics[width=.33\textwidth]{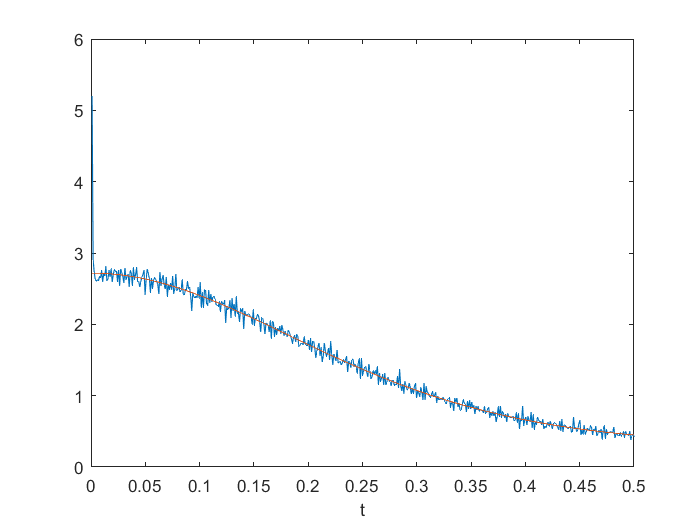} &
\includegraphics[width=.33\textwidth]{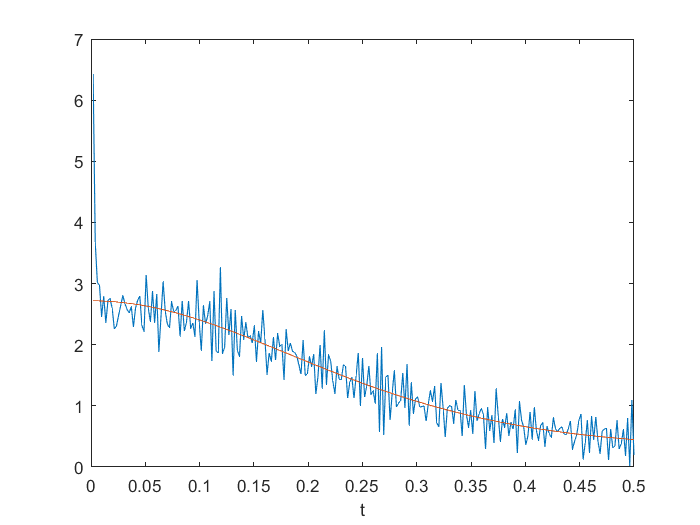} \\
\includegraphics[width=.33\textwidth]{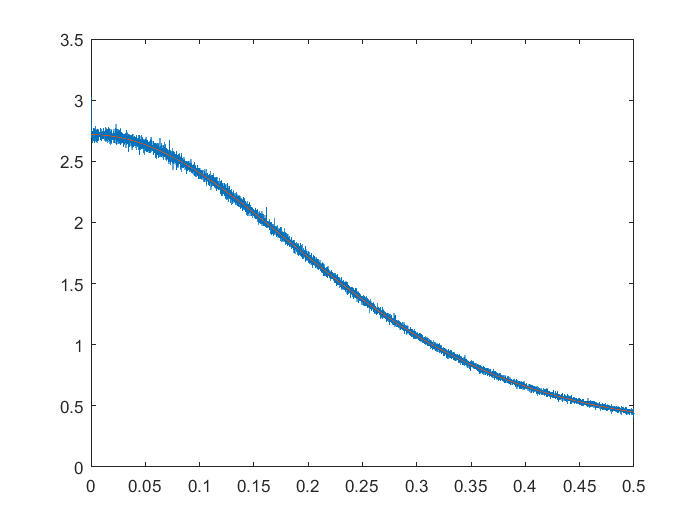} &
\includegraphics[width=.33\textwidth]{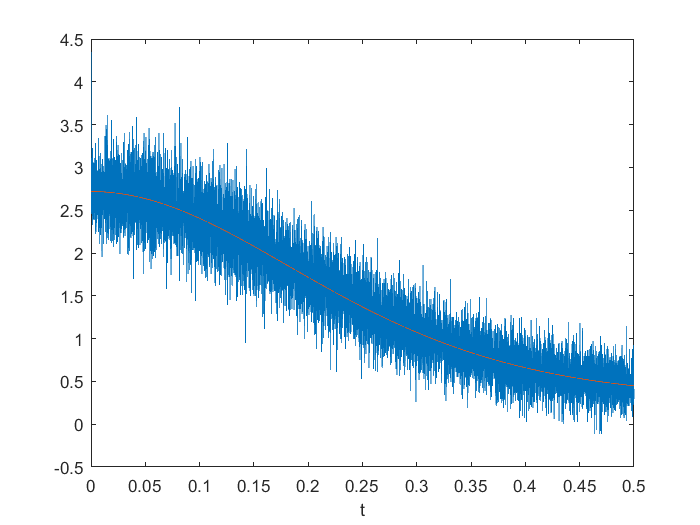} &
\includegraphics[width=.33\textwidth]{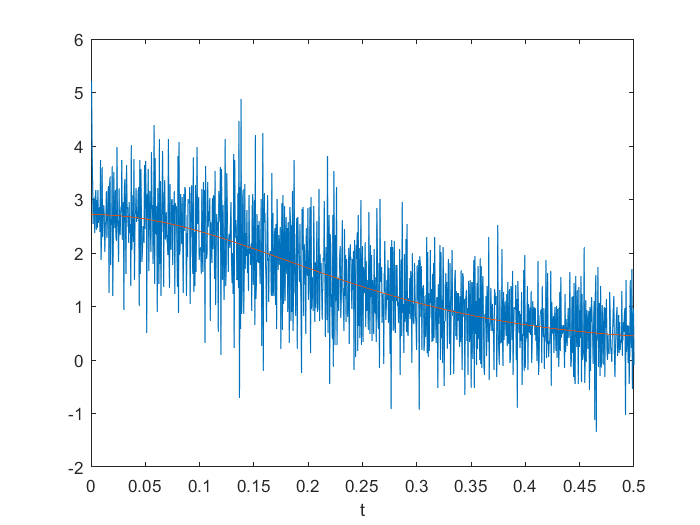} \\
(a) $\alpha=0.25$ & (b) $\alpha=0.5$ & (c) $\alpha=0.75$
\end{tabular}
\caption{The approximation with various time step size, for noisy data with $\epsilon=0.1\%$, 
at three fractional orders, $\alpha=0.25$, $\alpha=0.5$ and $\alpha=0.75$. Top row: 
$\ell^2$ error of the approximations vs. the time step size $\tau$. The next
three rows:  the reconstructions with different time discretization levels. From top to bottom, the total number $N$ of time steps 
is $2^{3}$, $2^{10}$ and $2^{13}$ for $\alpha=0.25$; $2^{2}$, $2^{9}$ and $2^{13}$ for 
$\alpha=0.5$; and  $2^{5}$, $2^{8}$ and $2^{10}$ for $\alpha=0.75$, the regularization is excessive (too large $\tau$), optimal (ideal $\tau$), and insufficient (too small $\tau$), in that order. The third row shows the reconstructions with the smallest error. \label{fig:semi-1d} }
\end{figure}

\section{Concluding remark}
In the current work, we focused on the reconstruction of a time-varying potential function in the time-fractional diffusion model from observations taken at a single point. By applying a set of reasonable assumptions to the data, we  derived a Lipschitz type conditional stability. Furthermore, drawing inspiration from the stability analysis, we proposed a iterative algorithm to approximately recover the potential and established a comprehensive error analysis of the discrete reconstruction, ensuring that the approximation error is congruent with the stability estimate we have established. Numerical tests were carried out to support and enhance the theoretical analysis.

Many interesting questions still remain open. For instance, recovering the spatially-dependent potential from a single-point observation presents a significant interest. Such problems are anticipated to exhibit stronger ill-posedness, and conducting their analysis would pose greater challenges, especially in terms of the numerical analysis of fully discrete schemes. Additionally, the identification of multiple coefficients using single or multiple observations presents a compelling avenue of research, each with its unique degree of ill-posedness. These interesting questions are left for future exploration.

\bibliographystyle{abbrv}

\end{document}